\documentclass[msom]{informs4}
\PassOptionsToPackage{table}{xcolor}
\RequirePackage{tgtermes}
\RequirePackage{newtxtext}
\RequirePackage{newtxmath}
\RequirePackage{bm}
\RequirePackage{endnotes}

\OneAndAHalfSpacedXII 

\usepackage{algorithm}
\usepackage{algpseudocode}

\usepackage{amsmath,amssymb,amsfonts}
\usepackage{mathtools}      

\usepackage{array}          
\usepackage{longtable}
\usepackage{booktabs}       

\usepackage{caption}

\usepackage{pifont}     
\usepackage{xcolor}     
\newcommand{\cmark}{\ding{51}} 

\usepackage[normalem]{ulem}
\newcommand{\add}[1]{\textcolor{blue}{#1}}

\usepackage[short]{optidef}

\newcolumntype{L}[1]{>{\raggedright\arraybackslash}p{#1}}

\makeatletter
\@ifundefined{SingleSpacedXI}{}{}
\makeatother

\usepackage{tikz}
\usetikzlibrary{shapes, arrows.meta, positioning, fit, backgrounds, shadows, calc, svg.path}

\usepackage{natbib}
 \bibpunct[, ]{(}{)}{,}{a}{}{,}%
 \def\bibfont{\small}%
\EquationsNumberedBySection 

\TheoremsNumberedThrough     
\ECRepeatTheorems  %

\begin{document}


\RUNAUTHOR{Guan, Hijazi, and Van Hentenryck}

\RUNTITLE{BDCG for End-to-End Supply Chain Planning in the Paper Industry}

\TITLE{End-to-End Supply Chain Planning in the Paper Industry Via Column Generation and Benders Decomposition}

\ARTICLEAUTHORS{
\AUTHOR{Changkun Guan,\textsuperscript{a} Amira Hijazi,\textsuperscript{a} Pascal Van Hentenryck\textsuperscript{a}}
\AFF{\textsuperscript{a}H. Milton Stewart School of Industrial \& Systems Engineering, Georgia Institute of Technology, Atlanta, Georgia}
} 
\ABSTRACT{%
\textbf{\textit{Problem definition:}} The paper studies an integrated end-to-end planning problem in large-scale paper manufacturing, where production scheduling, trimming decisions, vehicle loading, and multi-period fulfillment of make-to-order and make-to-stock demand must be coordinated over time. In practice, these decisions are often optimized sequentially, leading to material waste, inefficient transportation, and degraded service levels. Solving the fully integrated problem at industrial scale remains computationally challenging due to its combinatorial structure.
\textbf{\textit{Methodology / results:}} A key structural feature of the problem is that downstream fulfillment decisions
depend on upstream production and logistics choices only through aggregate supply availability over time. By exploiting this structure, the paper develops an exact mathematical formulation and proposes a two-phase hybrid framework (BDCG-DP) that integrates column generation (CG) using exact dynamic-programming (DP) for supply-side decisions with Benders decomposition (BD) for downstream fulfillment. Computational experiments on proprietary instances from a major North American paper manufacturer show that BDCG-DP lowers total costs by 24.4\% compared to a traditional CG-DP on challenging eight-week planning problems. Median runtime for four-week planning problems decreases from over five hours using CG-DP to under one hour using BDCG-DP.
\textbf{\textit{Managerial implications:}} This paper provides the first exact model that integrates production, trimming, load planning, and multi-period fulfillment at an industrial scale. The proposed approach returns integer-feasible plans within 2.3 to 6 hours for the most complex planning problems, enabling planners to access high-quality implementable schedules within hours, a capability that was previously unavailable in practice.}





\maketitle

\section{Introduction}
\label{sec:introduction}
In real-world manufacturing industries, operational decision-making spans multiple layers, including
production scheduling, inventory management, and transportation planning. These layers are
tightly interconnected and often involve conflicting objectives, making coordinated end-to-end
planning essential for operational efficiency and profitability. The paper industry is no exception,
requiring integrated decision-making across customer order allocation, production scheduling, trimming, load planning, and transportation. Furthermore, these decisions must accommodate different demand fulfillment strategies: make-to-stock (MTS), where forecast-driven production replenishes external, geographically spread warehouses to balance inventory costs and service levels; and make-to-order (MTO), where production is triggered by customer requests, demanding flexible scheduling and rapid
response.  Moreover, in practice, decisions are re-optimized weekly in a rolling-horizon planning framework, where an eight-week planning horizon is updated as new information becomes available. Planners routinely evaluate multiple what-if scenarios to assess the impacts of demand fluctuations, capacity disruptions, and alternative logistics strategies.

Despite this interdependence, important components of the planning process are still often addressed sequentially in practice. For example, in integrated lot-sizing and cutting stock settings, a common practice is to solve the lot-sizing problem first and then solve the cutting-stock problem on the resulting quantities \citep{GRAMANI2006509}. At an industrial scale, this could result into material waste, infeasible or inefficient loading plans, increased transportation costs, and degraded customer service, motivating integrated optimization approaches.

From a modeling perspective, existing studies largely address these planning components in a decentralized manner or with limited coordination: production and lot-sizing models abstract away detailed trimming and loading, cutting-stock models rarely capture downstream inventory and fulfillment dynamics, and recent multi-period fulfillment models omit the combinatorial trimming and loading structure inherent to paper manufacturing. To the best of the authors' knowledge, no prior approach captures the interaction among production, trimming, vehicle loading, and multi-period MTS/MTO fulfillment in a single industrial-scale exact formulation.

The planning problem considered here integrates production scheduling across mills, trimming decisions, transportation planning, and multi-period inventory management within a unified optimization framework. A key structural feature of the problem is that downstream fulfillment decisions depend on upstream production and logistics choices only through aggregate supply availability over time. This aggregation enables separating the complex combinatorial decisions associated with production, trimming, and loading and the downstream demand fulfillment using decomposition methods.

\subsection{Contributions}
The main contributions of this paper are as follows:
\begin{itemize}
  \item \textbf{An exact formulation of the end-to-end industrial-scale planning problem.} A multi-period exact formulation is developed for the integrated paper manufacturing and distribution planning problem, jointly capturing production, trimming-pattern selection, vehicle loading decisions, transportation, and multi-period demand fulfillment of MTS and MTO demand.
  
  \item \textbf{A structural decomposition and hybrid solution framework.} The paper identifies the aggregate supply profile indexed by product, customer, and period as the structural interface between the supply-side decisions and downstream fulfillment. Once this profile is fixed, the downstream MTS problem decomposes exactly by product-customer pair. Exploiting this structure yields a hybrid solution framework that applies column generation to the supply-side trimming and loading decisions and Benders decomposition to the downstream fulfillment process. Exact dynamic programs solve the supply-side pricing subproblems, and an additional integer refinement step is applied over the generated column pools to produce implementable plans.
  
  \item \textbf{Computational evidence of scalability at an industrial scale.} On 58 proprietary instances, the proposed framework delivers massive performance improvements across all planning horizons. For one- and two-week horizons, the proposed method reduces median solution times by factors of 16.3 and 8.4 relative  to a standard column-generation baseline. On the most challenging eight-week horizons, the framework consistently dominates by improving the final root-LP gap in every single instance, lowering total costs by 24.4\%. 

  \item \textbf{Managerial insights for implementable plans.} On the most complex eight-week instances, the proposed framework finds its first integer-feasible solution within 2.0 and 5.7 hours. The generated solutions have an optimality gap of less than 10\% relative to the root LP and achieve a lower objective value than the final plans produced by the baseline method. As a result, planners can access high-quality implementable schedules within hours, a capability that was previously unavailable in practice.

\end{itemize}

The remainder of this paper is organized as follows. Section~\ref{sec:literature} reviews the literature. Section~\ref{sec:description} introduces the problem in detail, while Section~\ref{sec:mathematical_formulation} presents the mathematical formulations. Section~\ref{sec:column_generation_bender} describes the solution framework. Section~\ref{sec:case_study} reports the computational study, and Section~\ref{sec:conclusion} concludes the paper. Additional computational evidence and proofs are provided in the Electronic Companion.

\section{Literature Review}
\label{sec:literature}
\subsection{Lot-Sizing and Cutting Stock}
The one-dimensional cutting stock problem (CSP) chooses cutting patterns that meet item demand with as little trim loss as possible, while the lot-sizing problem decides how much to produce in each period. The linear-programming view of production planning goes back to \citet{Kantorovich1960-blx}, and \citet{GilmoreGomory1961} introduced the column-generation formulation that makes the CSP tractable. Later work extends the basic model in directions that matter in practice: uncertain demand \citep{stochastictrim-loss2009}, due dates and rolling-horizon planning \citep{CSPduedate2010}, the reuse of leftover material across periods \citep{CSPUL2014,CSPULSurvey}, and multiple stock lengths or larger pattern sets handled with cutting planes and column generation \citep{BELOV2002274,branchcutprice2006}. Paper production further constrains the cutting decision through minimum production times and setup costs \citep{CSPNonlinear1971,diegel2006enforcing}. On the lot-sizing side, the multi-item capacitated problem with setup times and lost sales has been solved with mixed-integer heuristics and cutting-plane methods \citep{RO2007MIP,EJOR2008}, and \citet{Loparic2001} studies the uncapacitated case with sales and safety stocks. All of these models stop at the production or cutting stage and do not consider how the resulting rolls are loaded, shipped, and used to meet demand.

\subsection{Vehicle Loading and Distribution Planning}
Once rolls are cut, they must be combined into feasible vehicle loads and delivered to customers, which adds packing and transportation decisions on top of cutting. Building a load is itself a packing problem, in the spirit of the knapsack model studied by \citet{caprara2004two}. \citet{CSPTransport2007} link cutting directly to transportation, solving a combined cutting-and-transportation model with Benders decomposition and column generation. A separate line of work studies how to manage inventory and order fulfillment over time: using advance demand information \citep{chan2024abi}, sharing inventory flexibly across locations \citep{devalve2023fulfillment}, choosing when to commit to customer orders \citep{figueira2023commit}, and adjusting lead times in assemble-to-order systems \citep{yu2024clf}. These fulfillment models treat the available supply as given.

\subsection{Integrated Planning in the Paper Industry}
A long line of work in the paper industry links cutting with upstream production. \citet{GRAMANI2006509} weigh trim loss against setup and inventory cost in a combined cutting-and-lot-sizing model and show clear gains over solving the two problems separately, and \citet{CSPlotsizing2008} extend this coupling to several paper machines. Later studies improve the solution methods: \citet{leao2017decomposition} apply Dantzig--Wolfe decomposition and column generation, \citet{wu2017progressive} combine reformulation with heuristic search, \citet{pierini2021analysis} tighten the integrated formulation, and \citet{furlan2024matheuristic} add paper-grade sequencing, digester speeds, and chemical recovery to the model. Going beyond cutting and lot sizing, \citet{pinar2002} build a decision-support system that schedules paper production together with distribution. This body of work shows the value of coordinating production and cutting.

\subsection{Algorithmic Paradigms: Column Generation and Benders Decomposition}
The proposed solution method in this paper relies on two classical decomposition ideas: Dantzig--Wolfe decomposition and Column Generation \citep{DantzigWolfe1960, GilmoreGomory1961}, and Benders decomposition \citep{Benders1962}. Recent studies combine the two ideas. \citet{BenderCG2021} pair Benders decomposition with column generation for a cutting stock problem with time-varying demand, \citet{basciftci2024adaptive} use a related decomposition for multi-period stochastic planning, and \citet{mandl2023commodity} solve a multi-period planning problem with data-driven mixed-integer optimization. None of these methods, however, uses the aggregate supply profile as the link that separates the supply-side cutting and loading decisions from downstream fulfillment.

Table~\ref{tab:litgap} provides a feature-by-feature comparison of this paper and the existing literature. To the best of the authors' knowledge, no prior study combines production scheduling across mills, pattern-based cutting, vehicle loading, transportation, and multi-period make-to-stock and make-to-order fulfillment in one exact, industrial-scale, rolling-horizon model. 

\begin{table*}[t]
\centering
\caption{Positioning of this paper relative to the literature.}
\label{tab:litgap}
\setlength{\tabcolsep}{3pt}\renewcommand{\arraystretch}{1.1}\scriptsize
\begin{tabular}{@{}>{\raggedright\arraybackslash}p{3.05cm} *{7}{>{\centering\arraybackslash}p{0.95cm}} >{\raggedright\arraybackslash}p{3.35cm}@{}}
\toprule
 & \multicolumn{6}{c}{\textbf{Problem features addressed}} & & \\
\cmidrule(lr){2-7}
\textbf{Study} & \textbf{\shortstack{Lot-\\sizing}} & \textbf{\shortstack{Cutting\\stock}} & \textbf{\shortstack{Vehicle\\loading}} & \textbf{MTS} & \textbf{MTO} & \textbf{\shortstack{Rolling\\horizon}} & \textbf{\shortstack{Ind.\\scale}} & \textbf{Solution technique} \\
\midrule
\rowcolor{gray!20}\multicolumn{9}{@{}l}{\textbf{\textit{Lot-Sizing and Cutting Stock}}}\\
\citet{CSPNonlinear1971} &  & \cmark &  &  &  &  & \cmark & Heuristic methods \\
\rowcolor{gray!7} 
\citet{Loparic2001} & \cmark &  &  & \cmark &  &  &  & Extended formulations; dynamic program \\
\citet{BELOV2002274} &  & \cmark &  &  &  &  &  & Cutting planes; column generation \\
\rowcolor{gray!7} 
\citet{branchcutprice2006} &  & \cmark &  &  &  &  &  & Branch-and-cut-and-price \\
\citet{diegel2006enforcing} &  & \cmark &  &  &  &  & \cmark & Analytical feasibility/cost study \\
\rowcolor{gray!7} \citet{RO2007MIP} & \cmark &  &  &  &  &  &  & MIP-based heuristics \\
\citet{EJOR2008} & \cmark &  &  &  &  &  & \cmark & MIP heuristics; branch-and-cut \\
\rowcolor{gray!7} \citet{stochastictrim-loss2009} &  & \cmark &  &  &  &  &  & Two-stage stochastic programming \\
\citet{CSPduedate2010} &  & \cmark &  &  &  & \cmark & \cmark & Optimization models; rolling horizon \\
\rowcolor{gray!7} \citet{CSPUL2014} &  & \cmark &  & \cmark &  &  &  & Survey (usable leftovers) \\
\citet{CSPULSurvey} &  & \cmark &  & \cmark &  &  &  & Literature review \\
\rowcolor{gray!20}\multicolumn{9}{@{}l}{\textbf{\textit{Vehicle Loading and Distribution Planning}}}\\
\rowcolor{gray!7} \citet{caprara2004two} &  &  & \cmark &  &  &  &  & Two-dimensional knapsack algorithms \\
\citet{CSPTransport2007} &  & \cmark & \cmark &  &  &  & \cmark & Benders partitioning; column generation \\
\rowcolor{gray!7} \citet{chan2024abi} &  &  &  & \cmark &  &  & \cmark & Stochastic inventory; state aggregation \\
\citet{devalve2023fulfillment} &  &  &  & \cmark & \cmark &  & \cmark & Dynamic/stochastic optimization; simulation \\
\rowcolor{gray!7} \citet{figueira2023commit} &  &  &  & \cmark & \cmark &  & \cmark & Allocation policies; empirical \\
\citet{yu2024clf} &  &  &  & \cmark & \cmark & \cmark & \cmark & Dynamic programming \\
\rowcolor{gray!20}\multicolumn{9}{@{}l}{\textbf{\textit{Integrated Planning in the Paper Industry}}}\\
\rowcolor{gray!7} \citet{pinar2002} & \cmark &  &  &  &  &  & \cmark & Decision-support system; multicriteria \\
\citet{GRAMANI2006509} & \cmark & \cmark &  &  &  &  &  & Integrated model; shortest-path \\
\rowcolor{gray!7} \citet{CSPlotsizing2008} & \cmark & \cmark &  &  &  &  & \cmark & Coupled model; heuristics \\
\citet{leao2017decomposition} & \cmark & \cmark &  &  &  &  & \cmark & Dantzig--Wolfe; column generation; ALNS \\
\rowcolor{gray!7} \citet{wu2017progressive} & \cmark & \cmark &  &  &  &  &  & Progressive selection; reformulation \\
\citet{pierini2021analysis} & \cmark & \cmark &  &  &  &  &  & Formulation analysis; relax-and-fix \\
\rowcolor{gray!7} \citet{furlan2024matheuristic} & \cmark &  &  & \cmark &  &  & \cmark & Matheuristic; fix-and-optimize \\
\rowcolor{gray!20}\multicolumn{9}{@{}l}{\textbf{\textit{Algorithmic Paradigms: Column Generation and Benders Decomposition}}}\\
\citet{BenderCG2021} &  & \cmark &  &  &  &  & \cmark & Hybrid Benders + column generation \\
\rowcolor{gray!7} \citet{basciftci2024adaptive} &  &  &  &  &  & \cmark & \cmark & Adaptive two-stage stochastic programming \\
\midrule
\rowcolor{gray!28}\textbf{This paper} & \cmark & \cmark & \cmark & \cmark & \cmark & \cmark & \cmark & \textbf{Hybrid BDCG-DP} \\
\bottomrule
\end{tabular}
{\par\vspace{3pt}\raggedright\scriptsize \textit{Notes.} \cmark{} indicates the feature is addressed by the study; a blank cell indicates it is not. Lot-sizing $=$ multi-period lot sizing; MTS $=$ make-to-stock fulfillment and inventory; MTO $=$ make-to-order fulfillment; Rolling horizon $=$ periodic re-optimization over a rolling planning horizon as information is updated; Ind.\ scale $=$ industrial-scale instances. Studies are grouped by the research stream, as in Section~\ref{sec:literature}.\par}
\end{table*}

\section{The Integrated Paper Manufacturing and Distribution Scheduling Problem Description}
\label{sec:description}

This paper studies an integrated planning problem arising in large-scale paper manufacturing
systems that jointly optimizes production scheduling, trimming, transportation, and multi-period
inventory management over a finite planning horizon with $T$ periods. The problem is motivated
by industrial settings in which these decisions are highly interdependent yet are often addressed in
isolation, leading to material waste, infeasible logistics plans, or inefficient inventory utilization.

Production is carried out on a set of machines located across multiple mills. Within each planning period $t = 1, \dots, T$, the available capacity on a machine is partitioned into a sequence of \emph{production runs}, defined as contiguous time blocks during which a single paper grade (type) is produced. Each run  is bounded by maximum allowable duration and produces \emph{jumbo rolls}, which constitute large paper reels that are subsequently cut into finished products. The sequence of these runs is fixed a priori to reflect long-term planning decisions such as grade changeovers, setup considerations, and machine availability. Given this predetermined sequence, \textbf{production planning} determines the duration of each run and the number of
jumbo rolls produced, subject to machine capacity constraints. Due to physical storage limitations within the mill, \textit{all jumbo rolls produced during a run must be fully processed, and no inventory of jumbo rolls is permitted.}

Each jumbo roll has a fixed deckle width, whereas final  customer products require diverse widths. The \textbf{trimming}
stage consequently involves selecting feasible combinations of customer demands such that their aggregate width does not exceed the deckle width of the jumbo roll. When considered in isolation, this problem
corresponds to the classical one-dimensional knapsack problem  \citep{Kantorovich1960-blx, GilmoreGomory1961}, where the objective is to satisfy demand while minimizing trim loss. In the integrated setting studied here, however, trimming decisions are linked both to the production runs that create supply and to the loading decisions that distribute finished items downstream.

After trimming, finished items are distributed through a transportation network using a finite set of vehicle
types. In the make-to-stock (MTS) channel, shipments replenish external, geographically spread warehouses. In the MTO channel, shipments go directly to customers. \textbf{Transportation planning} determines how items are consolidated into vehicle loads subject to weight and loading constraints.  This
problem corresponds to the two-dimensional knapsack problem  \citep{caprara2004two}. \textit{Following the real-world practice, this paper assumes that items shipped on a vehicle must originate from the
same production run and be delivered to the same destination, so loading feasibility is coupled with the cutting decision.}

Finally, \textbf{inventory decisions} balance supply and demand over time and across locations for the
MTS channel, where finished items may be carried from one period to the next. In contrast, MTO
demand is strictly period-specific and cannot be backlogged across periods. Unmet demand for
both MTO and MTS customers, as well as inventory levels below prescribed safety stock targets
for MTS customers, result in penalties representing lost sales and service-level violations.

The model is used in a rolling-horizon setting consistent with industrial practice. In the case study, periods are weeks, and the horizon extends up to eight weeks. At each planning period, the model is reoptimized using updated demand forecasts, inventory positions, and operating conditions. Planners also use the model to evaluate scenarios involving demand fluctuations, capacity changes, and alternative transportation options.  As a result, the optimization model must be solved repeatedly and efficiently, further amplifying the challenges posed by the problem’s scale and combinatorial nature. Figure~\ref{fig:E2E_problem} provides a schematic view of the problem.
\begin{figure}[!t]
  \centering
  \resizebox{0.95\textwidth}{!}{
\begin{tikzpicture}[
  every node/.style={font=\sffamily},
  stagetitle/.style={font=\sffamily\bfseries\footnotesize, align=center, text width=3.0cm},
  stagedesc/.style={font=\sffamily\scriptsize, align=center, text width=3.2cm, anchor=north, execute at begin node={\hyphenpenalty=10000\exhyphenpenalty=10000\relax}},
  icon/.style={line cap=round, line join=round, draw=black},
  flow/.style={-{Stealth[length=2.6mm]}, line width=0.5mm, draw=black}
]
\node[stagetitle] at (0,1.65) {Production Planning};
\begin{scope}[shift={(0,0)}, yscale=-1, scale=1.85, shift={(-12pt,-12pt)}]
  \draw[icon, line width=1.3mm] svg {M3 21h18};
  \draw[icon, line width=1.3mm] svg {M5 21v-12l5 4v-4l5 4h4};
  \draw[icon, line width=1.3mm] svg {M19 21v-8l-1.436 -9.574a0.5 0.5 0 0 0 -0.495 -0.426h-1.145a0.5 0.5 0 0 0 -0.494 0.418l-1.43 8.582};
  \draw[icon, line width=1.3mm] svg {M9 17h1};
  \draw[icon, line width=1.3mm] svg {M14 17h1};
\end{scope}
\node[stagedesc] at (0,-1.15) {Schedule machine run times to produce jumbo rolls};

\node[stagetitle] at (3.4,1.65) {Trimming};
\begin{scope}[shift={(3.4,0)}, yscale=-1, scale=1.85, shift={(-12pt,-12pt)}]
  \draw[icon, line width=1.3mm] svg {M6 7m-3 0a3 3 0 1 0 6 0a3 3 0 1 0 -6 0};
  \draw[icon, line width=1.3mm] svg {M6 17m-3 0a3 3 0 1 0 6 0a3 3 0 1 0 -6 0};
  \draw[icon, line width=1.3mm] svg {M8.6 8.6l10.4 10.4};
  \draw[icon, line width=1.3mm] svg {M8.6 15.4l10.4 -10.4};
\end{scope}
\node[stagedesc] at (3.4,-1.15) {Cut each jumbo roll into product widths within the deckle width};

\node[stagetitle] at (6.8,1.65) {Finished Products};
\begin{scope}[shift={(6.8,0.42)}, yscale=-1, scale=1.19, shift={(-12pt,-12pt)}]
  \draw[icon, line width=0.85mm] svg {M5 6c0 -1.657 3.134 -3 7 -3s7 1.343 7 3s-3.134 3 -7 3s-7 -1.343 -7 -3};
  \draw[icon, line width=0.85mm] svg {M5 6v12c0 1.657 3.134 3 7 3s7 -1.343 7 -3v-12};
\end{scope}
\begin{scope}[shift={(6.35,-0.42)}, yscale=-1, scale=1.19, shift={(-12pt,-12pt)}]
  \draw[icon, line width=0.85mm] svg {M5 6c0 -1.657 3.134 -3 7 -3s7 1.343 7 3s-3.134 3 -7 3s-7 -1.343 -7 -3};
  \draw[icon, line width=0.85mm] svg {M5 6v12c0 1.657 3.134 3 7 3s7 -1.343 7 -3v-12};
\end{scope}
\begin{scope}[shift={(7.25,-0.42)}, yscale=-1, scale=1.19, shift={(-12pt,-12pt)}]
  \draw[icon, line width=0.85mm] svg {M5 6c0 -1.657 3.134 -3 7 -3s7 1.343 7 3s-3.134 3 -7 3s-7 -1.343 -7 -3};
  \draw[icon, line width=0.85mm] svg {M5 6v12c0 1.657 3.134 3 7 3s7 -1.343 7 -3v-12};
\end{scope}
\node[stagedesc] at (6.8,-1.15) {Finished products ready for shipment};

\node[stagetitle] at (10.2,1.65) {Transportation \& Logistics};
\begin{scope}[shift={(10.2,0)}, yscale=-1, scale=1.85, shift={(-12pt,-12pt)}]
  \draw[icon, line width=1.3mm] svg {M7 17m-2 0a2 2 0 1 0 4 0a2 2 0 1 0 -4 0};
  \draw[icon, line width=1.3mm] svg {M17 17m-2 0a2 2 0 1 0 4 0a2 2 0 1 0 -4 0};
  \draw[icon, line width=1.3mm] svg {M5 17h-2v-11a1 1 0 0 1 1 -1h9v12m-4 0h6m4 0h2v-6h-8m0 -5h5l3 5};
\end{scope}
\node[stagedesc] at (10.2,-1.15) {Build vehicle loads and dispatch shipments};

\node[stagetitle, text width=3.4cm] at (14.4,2.55) {Make-to-Stock (MTS)};
\begin{scope}[shift={(14.4,1.5)}, yscale=-1, scale=1.42, shift={(-12pt,-12pt)}]
  \draw[icon, line width=1.0mm] svg {M3 21v-13l9 -4l9 4v13};
  \draw[icon, line width=1.0mm] svg {M13 13h4v8h-10v-6h6};
  \draw[icon, line width=1.0mm] svg {M13 21v-9a1 1 0 0 0 -1 -1h-2a1 1 0 0 0 -1 1v3};
\end{scope}
\node[stagedesc, text width=3.4cm] at (14.4,0.7) {Shipments replenish external warehouses inventory};

\node[stagetitle, text width=3.4cm] at (14.4,-1.0) {Make-to-Order (MTO)};
\begin{scope}[shift={(14.4,-2.1)}, yscale=-1, scale=1.42, shift={(-12pt,-12pt)}]
  \draw[icon, line width=1.0mm] svg {M8 7a4 4 0 1 0 8 0a4 4 0 0 0 -8 0};
  \draw[icon, line width=1.0mm] svg {M6 21v-2a4 4 0 0 1 4 -4h4a4 4 0 0 1 4 4v2};
\end{scope}
\node[stagedesc, text width=3.4cm] at (14.4,-2.9) {Shipments go directly to customers};

\draw[flow] (1.15,0) -- (2.25,0);
\draw[flow] (4.55,0) -- (5.65,0);
\draw[flow] (7.95,0) -- (9.05,0);
\draw[flow] (11.35,0.25) -- (12.75,1.3);
\draw[flow] (11.35,-0.25) -- (12.75,-1.9);
\end{tikzpicture}}
  \caption[Integrated End-to-End Supply Chain Problem]{Integrated end-to-end supply chain problem: production runs create jumbo rolls, trimming cuts them into finished products, vehicle loading and transportation deliver them, and shipments serve make-to-stock warehouses and make-to-order customers}.
  \label{fig:E2E_problem}
\end{figure}
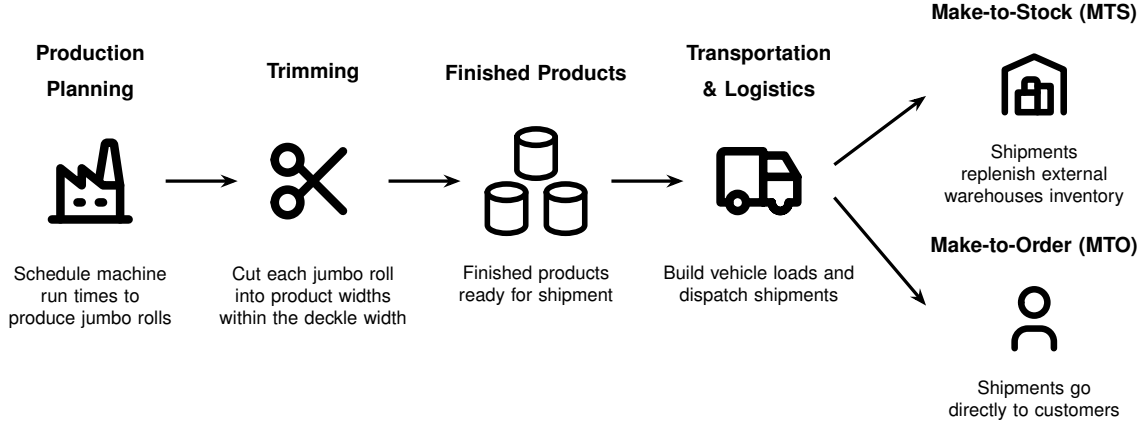
\section{Mathematical Formulation}
\label{sec:mathematical_formulation}

To model the integrated paper manufacturing and distribution scheduling problem, the formulation is developed in two steps. An item-based formulation is first introduced to explicitly map the physical and operational constraints at the level of individual jumbo rolls and vehicles. While exact, this formulation is highly symmetric and computationally intractable for industrial-scale instances. To overcome these computational hurdles, a pattern-based reformulation is subsequently presented. 
Table~\ref{aix_tab:notation_new} presents the core parameters and sets utilized across both formulations.

\begin{table}[!t]
  \centering
  \caption{Summary of notation}
  \label{aix_tab:notation_new}
  {\small
  \begin{tabular}{@{}L{0.20\textwidth}L{0.78\textwidth}@{}}
    \toprule
    \textbf{Symbol} & \textbf{Description} \\
    \midrule
    \multicolumn{2}{l}{\textbf{Sets and indices}}\\
    $M$ & Machines (index $m$)\\
    $T$ & Planning periods (index $t$)\\
    $P$ & Finished products (index $p$)\\
    $C^{\text{MTS}}$ & MTS customers (index $c$)\\
    $C^{\text{MTO}}$ & MTO customers (index $c$)\\
    $C$ & Customers, $C = C^{\text{MTS}} \cup C^{\text{MTO}}$\\
    $B$ & Production runs (index $b$)\\
    $B_{m,t}$ & Runs on machine $m$ in period $t$\\
    $B_t$ & Runs producing in period $t$, $B_t = \bigcup_{m \in M} B_{m,t}$\\
    $O_b$ & Potential jumbo rolls in run $b$ (index $o$)\\
    $V$ & Vehicle types (index $v$)\\
    $J_{b,c,v}$ & Potential vehicles of type $v$ from run $b$ to customer $c$ (index $j$)\\
    $K_b$ & Set of feasible trimming patterns for run $b$ \\
    $\Gamma_{b,c,v}$ & Set of feasible loading configurations for a vehicle of type $v$ dispatched from run $b$ to customer $c$\\
    $N$ & Segments for the piecewise-linear safety-stock penalty (index $i$)\\[4pt]

    \multicolumn{2}{l}{\textbf{Parameters}}\\
    $A_{m,t}$ & Available production capacity (hours) for machine $m$ in period $t$\\
    $r^{\max}_b$ & Maximum duration (hours) for run $b$\\
    $h_b$ & Hours required to produce one jumbo roll in run $b$\\
    $\mathrm{cost}^{\text{run}}_b$ & Run cost rate (per hour)\\
    $W_b$ & Deckle width of jumbo rolls in run $b$\\
    $w_p$ & Finished width of product $p$\\
    $g_{b,p}$ & Weight per unit of product $p$ produced in run $b$\\
    $q_p$ & Loading positions required per unit of product $p$\\
    $x^{\min}_{b,p,c,v}$ & Minimum shipment of product $p$ from run $b$ to customer $c$ via type $v$\\
    $G_v$ & Weight capacity per vehicle of type $v$\\
    $Q_v$ & Position capacity per vehicle of type $v$\\
    $\mathrm{cost}^{\text{veh}}_{b,c,v}$ & Transportation cost per vehicle of type $v$ from run $b$ to customer $c$\\
    $d_{p,c,t}$ & Demand for product $p$ at customer $c$ in period $t$\\
    $I^{\text{init}}_{p,c}$ & Initial inventory at MTS customer $c$\\
    $\mathrm{cost}^{\text{lost}}_{p,c}$ & Lost-sales penalty per unit of product $p$ at customer $c$\\
    $ss_{p,c,t}$ & Safety-stock target\\
    $\alpha_{p,c,t,i}$, $\beta_{p,c,t,i}$ & Piecewise-linear safety-stock penalty coefficients\\[4pt]
    \bottomrule
  \end{tabular}
  }
\end{table}

\subsection{Item-Based Formulation}\label{sec:item_based_model}

The item-based formulation explicitly represents production, trimming, and transportation decisions
at the level of finished products. In particular, jumbo rolls are objects that
must be produced, cut into finished products, and subsequently loaded onto vehicles for delivery. This level of detail results in a very large mixed-integer formulation. As such, the item-based
model is primarily intended as a reference formulation and as a foundation for the decomposition
approach introduced later, rather than as a model to be solved directly at an industrial scale. 

Production is organized into a predefined set of runs $b\in B$, each corresponding to a machine, a paper grade, and a time period. Since the exact number of jumbo rolls in a run is not known in advance, the model introduces for each run $b$ an upper bound on the rolls that could be produced in $b$ defined by the maximum allowable run duration, $r^{\max}_b$,  divided by the number of production hours needed per jumbo roll, $h_b$. Thus, the set of potential jumbo rolls is given by
\[ 
O_b=\{1,\ldots,\lfloor r^{\max}_b/h_b \rfloor\}. 
\]
The production and trimming decisions are represented by the following variables:
\begin{itemize}
    \item $r_b \ge 0$: duration (hours) of run $b$
    \item $u_{b,o} \in \{0,1\}$: 1 if jumbo roll $o$ is produced in run $b$
    \item $x_{o,p,c,v} \in \mathbb{Z}_+$: units of product $p$ cut from roll $o$ and allocated to customer $c$ for shipment using vehicle type $v$
\end{itemize}
Transportation is modeled at the level of individual vehicle assignment. A vehicle type (e.g., truck or railcar) is indexed by $v\in V$, with weight capacity $G_v$ and position capacity $Q_v$. Consistent with operational
practice, each vehicle carries products originating from a single production run and destined for a
single customer.
Thus, for each combination of run $b$, customer $c$, and vehicle type $v$, the model introduces a finite set of potential
vehicles $J_{b,c,v}$ such that:
\[
  |J_{b,c,v}|
  \;=\;
  \left\lfloor \frac{r^{\max}_b}{h_b} \right\rfloor
  \left\lfloor \frac{W_b}{\underline w_b} \right\rfloor,
  \qquad
  \underline w_b := \min\{w_p : w_p \le W_b\},
\]
which provides a conservative upper bound based on the maximum number of products that can be produced in run $b$.
 
 The transportation decisions are represented by:
\begin{itemize}   
    \item $\chi_{b,p,c,v,j} \in \mathbb{Z}_+$: units of product $p$ loaded onto vehicle $j \in J_{b,c,v}$,
    \item $z_{b,c,v,j} \in \{0,1\}$: equals 1 if vehicle $j \in J_{b,c,v}$ is dispatched
\end{itemize}

 The downstream fulfillment decisions are represented by:
\begin{itemize}   
    \item $I^{+}_{p,c,t} \in \mathbb{Z}_+$: end-of-period inventory for MTS customers,
    \item $l_{p,c,t} \in \mathbb{Z}_+$: lost sales, with $0 \le l_{p,c,t} \le d_{p,c,t}$,
    \item $I^{-}_{p,c,t} \in \mathbb{Z}_+$: safety-stock shortfall,
    \item $\psi_{p,c,t} \ge 0$: safety-stock penalty.
\end{itemize}

\subsubsection{Objective Function}
The objective is to minimize total cost comprising production cost $(\mathrm{cost}^{\text{run}}_b)$, transportation costs $(\mathrm{cost}^{\text{veh}}_{b,c,v})$,
lost-sales penalties $(\mathrm{cost}^{\text{lost}}_{p,c})$, and safety-stock violation penalties $(\psi_{p,c,t})$ as follows:
\begin{equation}
\label{aix_eq:obj_new}
\begin{aligned}
  \min \quad 
  & \sum_{b \in B} \mathrm{cost}^{\text{run}}_b\,r_b
    + \sum_{b \in B}\sum_{c \in C}\sum_{v \in V}\sum_{j \in J_{b,c,v}}
        \mathrm{cost}^{\text{veh}}_{b,c,v}\,z_{b,c,v,j} \\
  & + \sum_{p \in P}\sum_{c \in C}\sum_{t \in T} \mathrm{cost}^{\text{lost}}_{p,c}\,l_{p,c,t} 
    + \sum_{p \in P}\sum_{c \in C^{\text{MTS}}}\sum_{t \in T} \psi_{p,c,t}
\end{aligned}
\end{equation}
\subsubsection{Production Constraints}
Production feasibility is enforced by the following constraints:
\begin{subequations}
\allowdisplaybreaks
\begin{align}
  & \sum_{b \in B_{m,t}} r_b \le A_{m,t}
    && \forall\, m \in M,\ t \in T \label{aix_eq:cap_machine} \\
  & r_b \le r^{\max}_b
    && \forall\, b \in B \label{aix_eq:block_bound} \\
  & r_b \ge h_b \sum_{o \in O_b} u_{b,o}
    && \forall\, b \in B \label{aix_eq:run_vs_rolls}
\end{align}
\end{subequations}
The production capacity constraint of the machines for each time period is given by \eqref{aix_eq:cap_machine}.
Constraint \eqref{aix_eq:block_bound} limits the operating time of each production run to its maximum allowable duration, $r^{\max}_b$. Finally, Constraint \eqref{aix_eq:run_vs_rolls} ensures that $r_b$ covers the total production time of the jumbo rolls produced in $b$.

\subsubsection{Cutting and Loading Constraints}
For each run $b$
and jumbo roll $o\in O_b$, the variable $x_{o,p,c,v}$ gives the number of units of product $p$ cut from roll $o$ and
allocated to customer $c$ for delivery via vehicle type $v$. 
Cutting feasibility is enforced by:
\begin{subequations}
\allowdisplaybreaks
\begin{align}
  & \sum_{p \in P}\sum_{c \in C}\sum_{v \in V}
      w_p\,x_{o,p,c,v}
    \le W_b\,u_{b,o}
    && \forall\, b \in B,\ o \in O_b \label{aix_eq:deckle_cap} 
    \end{align}
\end{subequations}
Individual vehicle feasibility is enforced through the potential-vehicle sets $J_{b,c,v}$. The binary variable
$z_{b,c,v,j}$ indicates whether vehicle $j\in J_{b,c,v}$ is dispatched from run $b$ to customer $c$ using vehicle type $v$,
and $\chi_{b,p,c,v,j}$ denotes the number of units of product $p$ loaded onto that vehicle. The loading constraints are:   
\begin{subequations}
\allowdisplaybreaks
\begin{align}
  & \sum_{o \in O_b} x_{o,p,c,v}
    = \sum_{j \in J_{b,c,v}} \chi_{b,p,c,v,j}
    && \forall\, b \in B,\ p \in P,\ c \in C,\ v \in V \label{aix_eq:load_link} \\
  & \sum_{p \in P} g_{b,p}\,\chi_{b,p,c,v,j} \le G_v\,z_{b,c,v,j}
    && \forall\, b \in B,\ c \in C,\ v \in V,\ j \in J_{b,c,v} \label{aix_eq:veh_wt_ind} \\
  & \sum_{p \in P} q_p\,\chi_{b,p,c,v,j} \le Q_v\,z_{b,c,v,j}
    && \forall\, b \in B,\ c \in C,\ v \in V,\ j \in J_{b,c,v} \label{aix_eq:veh_pos_ind} \\
     & \sum_{o \in O_b} x_{o,p,c,v} \ge x^{\min}_{b,p,c,v}
    && \forall\, b \in B,\ p \in P,\ c \in C,\ v \in V \label{aix_eq:min_ship}
\end{align}
\end{subequations}
Constraint
\eqref{aix_eq:load_link} links cutting and loading. Constraints
\eqref{aix_eq:veh_wt_ind}--\eqref{aix_eq:veh_pos_ind} enforce per-vehicle weight and position capacities using the unit weight
$g_{b,p}$ and per-unit position requirement $q_p$.  Additionally, minimum shipment requirements are enforced by Constraint \eqref{aix_eq:min_ship}.

\subsubsection{Downstream Fulfillment Constraints}
The final set of constraints governs demand fulfillment over the planning horizon. In both channels, lost sales are bounded above by the period's demand they represent. This upper bound is especially important for the MTS channel because lost sales enter the inventory recursion; without it, the model could create artificial inventory by recording lost sales in excess of actual demand.

For \textbf{Make-to-Order (MTO)} customers, demand is period-specific and cannot be backlogged across periods:
\begin{equation}
  \sum_{b \in B_t}\sum_{o \in O_b}\sum_{v \in V}
      x_{o,p,c,v} + l_{p,c,t}
    \ge d_{p,c,t}
    \quad \forall\, p \in P,\ c \in C^{\mathrm{MTO}},\ t \in T \label{aix_eq:mto_balance}
\end{equation}
The lost-sales variables satisfy
\begin{equation}
  0 \le l_{p,c,t} \le d_{p,c,t}
  \quad \forall\, p \in P,\ c \in C,\ t \in T \label{aix_eq:lost_sales_cap}
\end{equation}
Constraint \eqref{aix_eq:mto_balance} allows unmet demand to be captured as lost sales for MTO customers, and \eqref{aix_eq:lost_sales_cap} ensures that lost sales do not exceed the corresponding period demand.

For \textbf{Make-to-Stock (MTS)} customers, the end-of-period inventory flow balance constraints are as follows:
\begin{subequations}
\allowdisplaybreaks
\begin{align}
  I^{+}_{p,c,1} &= I^{\text{init}}_{p,c} - d_{p,c,1} + l_{p,c,1}
    + \sum_{b \in B_1}\sum_{o \in O_b}\sum_{v \in V}
          x_{o,p,c,v}
    && \forall\, p \in P,\ c \in C^{\mathrm{MTS}}
      \label{aix_eq:mts_balance-init} \\
  I^{+}_{p,c,t} &= I^{+}_{p,c,t-1} - d_{p,c,t} + l_{p,c,t}
    + \sum_{b \in B_t}\sum_{o \in O_b}\sum_{v \in V}
          x_{o,p,c,v}
    && \forall\, p \in P,\ c \in C^{\mathrm{MTS}},\ t > 1
      \label{aix_eq:mts_balance-rec}
\end{align}
\end{subequations}

Finally, for MTS customers, safety-stock constraints are 
\begin{subequations}
\allowdisplaybreaks
\begin{align}
  & I^{-}_{p,c,t} \ge ss_{p,c,t} - I^{+}_{p,c,t}
    && \forall\, p \in P,\ c \in C^{\mathrm{MTS}},\ t \in T
      \label{aix_eq:ss_def} \\
  & \psi_{p,c,t} \ge \alpha_{p,c,t,i}\, I^{-}_{p,c,t} - \beta_{p,c,t,i}
    && \forall\, p \in P,\ c \in C^{\mathrm{MTS}},\ t \in T,\ i \in N
      \label{aix_eq:ss_cost_cuts}
\end{align}
\end{subequations}
Constraint \eqref{aix_eq:ss_def} defines $I^{-}_{p,c,t}$ as the nonnegative deviation below the safety stock level. Under the same discrete inventory convention used for demand and inventory quantities, this shortfall is measured in integer units in the reference model. The variable
$\psi_{p,c,t}$ represents the corresponding penalty cost and is modeled as a convex piecewise-linear function of
$I^{-}_{p,c,t}$. Since $\psi_{p,c,t}$ is minimized in the objective, constraints \eqref{aix_eq:ss_cost_cuts} enforce the
epigraph representation of this convex penalty. Operationally, the penalty data are specified on a discrete inventory scale, so
the equivalent breakpoint representation has integral breakpoints. 

\subsection{Pattern-Based Reformulation}
\label{sec:pattern_based_model}
The item-based formulation is exact but highly symmetric: jumbo rolls within a run are interchangeable, and so are vehicles of the same type on a lane. The pattern-based reformulation preserves exactness while aggregating those interchangeable objects. It replaces roll-level trimming  decisions with trimming patterns and vehicle-level loading decisions with loading configurations.

Let $K_b$ denote the set of feasible trimming patterns for run $b$. A pattern $k \in K_b$ describes how one jumbo roll is cut, specifying the quantity $a_{p,c,v,k}$ of product $p$ allocated to customer $c$ for delivery via vehicle type $v$. Similarly, let $\Gamma_{b,c,v}$ denote the set of feasible loading configurations for a vehicle of type $v$ dispatched from run $b$ to customer $c$. Each configuration $\gamma \in \Gamma_{b,c,v}$ represents one feasible vehicle load on that lane through the quantities $a^{\text{load}}_{b,p,c,v,\gamma}$ of each product $p$ that the vehicle can carry. Feasibility requires that the total weight and position requirements of the loaded products do not exceed the vehicle capacities $G_v$ and $Q_v$. The decision variables for the pattern-based reformulations are as follows: 
\begin{itemize}
    \item $y^{\text{cut}}_{b,k} \in \mathbb{Z}_+$: Number of jumbo rolls in run $b$ processed according to cutting pattern $k \in K_b$.
    \item $y^{\text{load}}_{b,c,v,\gamma} \in \mathbb{Z}_+$: Number of vehicles of type $v$ from run $b$ to customer $c$ packed according to configuration $\gamma \in \Gamma_{b,c,v}$.
    \item $S_{p,c,t}$: Auxiliary \textit{supply profile} defining the quantity of product $p$ shipped to customer $c$ in period $t$, aggregated over all production runs and vehicle types in that period; an unrestricted continuous variable fixed by constraint~\eqref{eq:supply_profile_def}.
\end{itemize}

The pattern-based master problem is:
\begin{subequations} 
\label{form:MP}
\allowdisplaybreaks
\begin{align}
  \min \quad 
  &  \sum_{b \in B} \mathrm{cost}^{\text{run}}_b\,r_b
  + \sum_{b \in B}\sum_{c \in C}\sum_{v \in V} \sum_{\gamma \in \Gamma_{b,c,v}}
      \mathrm{cost}^{\text{veh}}_{b,c,v} \, y^{\text{load}}_{b,c,v,\gamma} \nonumber \\
  & + \sum_{p \in P}\sum_{c \in C}\sum_{t \in T} \mathrm{cost}^{\text{lost}}_{p,c}\,l_{p,c,t} 
  + \sum_{p \in P}\sum_{c \in C^{\text{MTS}}}\sum_{t \in T} \psi_{p,c,t} && \label{eq:MP_obj} \\[6pt]
  \text{s.t.} \quad 
  & \eqref{aix_eq:cap_machine},\;\eqref{aix_eq:block_bound}
  && \text{(machine capacity, run bound)} \nonumber \\
  & r_b \ge h_b \sum_{k \in K_b} y^{\text{cut}}_{b,k} 
  && \forall\, b \in B \quad (\mu_b) \label{eq:pat_run_vs_patterns} \\
  & \sum_{\gamma \in \Gamma_{b,c,v}} a^{\text{load}}_{b,p,c,v,\gamma}\,y^{\text{load}}_{b,c,v,\gamma}
  \ge \sum_{k \in K_b} a_{p,c,v,k}\,y^{\text{cut}}_{b,k} 
  && \forall\, b,\, p,\, c,\, v \quad (\lambda_{b,p,c,v}) \label{eq:pat_coupling} \\
  & \sum_{k \in K_b} a_{p,c,v,k}\,y^{\text{cut}}_{b,k} \ge x^{\min}_{b,p,c,v}
  && \forall\, b,\, p,\, c,\, v \quad (\eta_{b,p,c,v}) \label{eq:pat_min_ship} \\
  & S_{p,c,t} = \sum_{b \in B_t} \sum_{k \in K_b} \sum_{v \in V} a_{p,c,v,k}\,y^{\text{cut}}_{b,k}
  && \forall\, p,\, c,\, t \quad (\rho_{p,c,t}) \label{eq:supply_profile_def} \\
  & S_{p,c,t} + l_{p,c,t} \ge d_{p,c,t}
  && \forall\, p,\, c \in C^{\mathrm{MTO}},\, t  \label{eq:pat_mto} \\
  & I^{+}_{p,c,t} = I^{+}_{p,c,t-1} - (d_{p,c,t} - l_{p,c,t}) + S_{p,c,t}
  && \forall\, p,\, c \in C^{\mathrm{MTS}},\, t  \label{eq:pat_mts} \\
  & 0 \le l_{p,c,t} \le d_{p,c,t}
  && \forall\, p,\, c,\, t \label{eq:pat_lost_cap} \\
  & \eqref{aix_eq:ss_def},\;\eqref{aix_eq:ss_cost_cuts}
  && \text{(safety-stock penalty)} \nonumber
\end{align}
\end{subequations}

Constraint~\eqref{eq:pat_run_vs_patterns} enforces the run time implied by the selected trimming patterns. Constraint~\eqref{eq:pat_coupling} ensures sufficient  transportation capacity on each lane. Constraint~\eqref{eq:pat_min_ship} carries forward the minimum-shipment obligations from the item-based model; it is imposed on the cutting side because cut quantities determine realized shipments. Constraint~\eqref{eq:supply_profile_def} defines the aggregate supply profile used by the downstream MTO and MTS balance equations, and \eqref{eq:pat_lost_cap} carries forward the exact upper bound on lost sales. The convention $I^{+}_{p,c,0}:=I^{\text{init}}_{p,c}$ is adopted.

The Electronic Companion records the finite-index exactness argument for this reformulation. At a high level, the reformulation aggregates interchangeable jumbo rolls and same-type vehicles into counts of identical trimming patterns and loading configurations. On the loading side, \eqref{eq:pat_coupling} guarantees that the aggregate reserved productwise capacity on each lane is large enough to accommodate the realized cut quantities assigned to that lane.

\section{BDCG: A Hybrid Column-Generation and Benders-Decomposition Framework}
\label{sec:column_generation_bender}
The pattern-based reformulation in Section~\ref{sec:pattern_based_model} introduces two main computational challenges. First, the trimming pattern sets $K_b$ and loading configuration sets $\Gamma_{b,c,v}$ are prohibitively large and cannot be enumerated
explicitly. Second,  the make-to-stock (MTS) demand channel induces a multi-period inventory problem whose value must be reevaluated whenever the supply plan changes. A key structural feature is that the downstream MTS cost depends on the supply-side decisions only through the aggregate supply profile $S_{p,c,t}$. Once $S_{p,c,t}$ is fixed, the downstream problem decomposes by product-customer chain and can be evaluated independently. This structure naturally suggests a decomposition approach that separates supply-side combinatorial decisions from downstream inventory optimization. This section develops a solution framework that exploits this structure by coordinating supply-side column generation with demand-side Benders decomposition. The proposed \emph{Benders Decomposition with Column Generation} framework is referred to as BDCG hereafter; the implementation that solves the pricing subproblems by dynamic programming is denoted BDCG-DP, and the variant with MIP-based pricing is denoted BDCG-MIP.

BDCG adopts a two-phase implementation. First, a linear program (LP) construction phase, shown in Figure~\ref{fig:phase1_workflow}, alternates RMP optimization, supply-side pricing, and demand-side cut generation until no source of LP improvement remains. Phase II Figure~\ref{fig:phase2_workflow}, then fixes the resulting column and cut pools and restores integrality on the retained column-usage variables, yielding an integer Benders refinement loop over the fixed pools. This two-phase design trades a global optimality certificate for tractability: the returned solution is optimal over the Phase-1 column pools (up a tolerance), and its quality is measured against the exact root-LP lower bound established in Phase~1. 

\subsection{Structural Decomposition and Two-Phase Framework}
Let $\mathbf{S}=(\mathbf{S}_{p,c})_{p\in P,\ c\in C^{\mathrm{MTS}}}$ denote a fixed feasible aggregate MTS supply profile, where $\mathbf{S}_{p,c}=(S_{p,c,t})_{t\in T}$, and let $Z^{\mathrm{LP}}_{\mathrm{MTS}}(\mathbf{S})$ denote the resulting minimum downstream MTS LP cost.

\begin{proposition}[Supply-profile decomposition]
\label{prop:supply_profile_decomposition}
For any fixed feasible aggregate MTS supply profile $\mathbf{S}$,
\[
  Z^{\mathrm{LP}}_{\mathrm{MTS}}(\mathbf{S})
  =
  \sum_{p\in P}\sum_{c\in C^{\mathrm{MTS}}} Z^{\mathrm{LP}}_{p,c}(\mathbf{S}_{p,c}).
\]
\end{proposition}

Proposition~\ref{prop:supply_profile_decomposition} yields the decomposition on which BDCG is built.

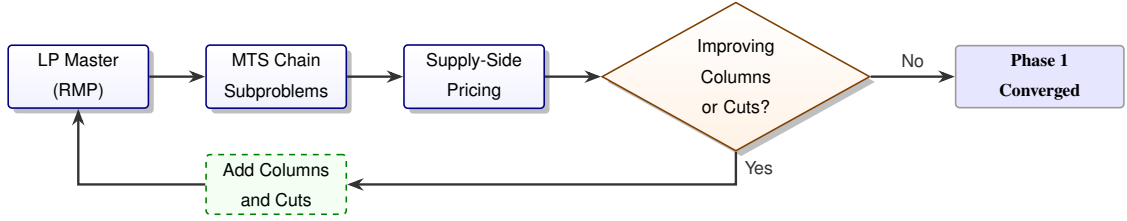
\begin{figure}[!t]
\centering
\resizebox{0.9\textwidth}{!}{%
\begin{tikzpicture}[
    node distance=1cm and 1cm,
    auto,
    >=Stealth,
    thick,
    font=\footnotesize\sffamily,
    process/.style={
        rectangle, 
        draw=blue!50!black, 
        top color=white, 
        bottom color=blue!5, 
        rounded corners=2pt, 
        minimum width=2.5cm, 
        minimum height=1cm, 
        align=center,
        drop shadow
    },
    decision/.style={
        diamond, 
        draw=orange!50!black, 
        top color=white, 
        bottom color=orange!10, 
        aspect=1.8, 
        inner sep=1pt, 
        text width=2.2cm, 
        align=center, 
        drop shadow
    },
    update/.style={
        rectangle, 
        draw=green!50!black, 
        fill=green!5,
        rounded corners=2pt, 
        minimum width=2.5cm, 
        minimum height=0.8cm, 
        align=center,
        dashed
    },
    terminal/.style={
        rectangle, 
        draw=gray!80, 
        fill=gray!10, 
        rounded corners=2pt, 
        minimum width=3cm, 
        minimum height=1cm, 
        align=center,
        font=\footnotesize\bfseries
    },
    line/.style={draw, ->, very thick, color=black!80}
]

    \node[process] (p1_master) {LP Master\\(RMP)};
    \node[process, right=of p1_master] (p1_benders) {MTS Chain\\Subproblems};
    \node[process, right=of p1_benders] (p1_pricing) {Supply-Side\\Pricing};
    \node[decision, right=of p1_pricing] (p1_check) {Improving\\Columns or Cuts?};
    
    \node[update, below=0.8cm of p1_benders] (p1_update) {Add Columns\\and Cuts};
    
    \node[terminal, right=1.5cm of p1_check, fill=blue!10] (p1_end) {Phase 1\\Converged};

    \draw[line] (p1_master) -- (p1_benders);
    \draw[line] (p1_benders) -- (p1_pricing);
    \draw[line] (p1_pricing) -- (p1_check);
    
    \draw[line] (p1_check.south) |- node[near start, right] {Yes} (p1_update.east);
    \draw[line] (p1_update.west) -| (p1_master.south);
    
    \draw[line] (p1_check.east) -- node[above] {No} (p1_end.west);
\end{tikzpicture}
}
\caption{Phase~1: LP relaxation with simultaneous column and cut generation}
\label{fig:phase1_workflow}
\end{figure}


\begin{figure}[!t]
\centering
\resizebox{0.85\textwidth}{!}{%
\begin{tikzpicture}[
    node distance=1cm and 1cm,
    auto,
    >=Stealth,
    thick,
    font=\footnotesize\sffamily,
    process/.style={
        rectangle, 
        draw=blue!50!black, 
        top color=white, 
        bottom color=blue!5, 
        rounded corners=2pt, 
        minimum width=2.5cm, 
        minimum height=1cm, 
        align=center,
        drop shadow
    },
    decision/.style={
        diamond, 
        draw=orange!50!black, 
        top color=white, 
        bottom color=orange!10, 
        aspect=1.8, 
        inner sep=1pt, 
        text width=2.2cm, 
        align=center, 
        drop shadow
    },
    update/.style={
        rectangle, 
        draw=green!50!black, 
        fill=green!5,
        rounded corners=2pt, 
        minimum width=2.5cm, 
        minimum height=0.8cm, 
        align=center,
        dashed
    },
    terminal/.style={
        rectangle, 
        draw=blue!50!black, 
        fill=blue!20, 
        rounded corners=2pt, 
        minimum width=2.5cm, 
        minimum height=1cm, 
        align=center,
        font=\bfseries,
        drop shadow
    },
    start_node/.style={
        rectangle,
        draw=gray!50,
        dashed,
        fill=gray!5,
        font=\scriptsize\itshape,
        align=center,
        minimum height=1cm
    },
    line/.style={draw, ->, very thick, color=black!80}
]

    \node[start_node] (p2_start) {Input from Phase 1:\\Fixed column pools\\and cut pools};
    \node[process, right=of p2_start] (p2_master) {Integer Master};
    \node[process, right=of p2_master] (p2_benders) {MTS Chain\\Subproblems};
    \node[decision, right=of p2_benders] (p2_check) {Converged?};
    
    \node[update, below=0.8cm of p2_benders] (p2_update) {Add Optimality\\Cuts};
    
    \node[terminal, right=of p2_check] (final) {Final\\Solution};

    \draw[line] (p2_start) -- (p2_master);
    \draw[line] (p2_master) -- (p2_benders);
    \draw[line] (p2_benders) -- (p2_check);
    
    \draw[line] (p2_check.south) |- node[near start, right] {No} (p2_update.east);
    \draw[line] (p2_update.west) -| (p2_master.south);
    
    \draw[line] (p2_check.east) -- node[above] {Yes} (final);

\end{tikzpicture}
}
\caption{Phase~2: Integer Benders refinement over fixed column pools}
\label{fig:phase2_workflow}
\end{figure}
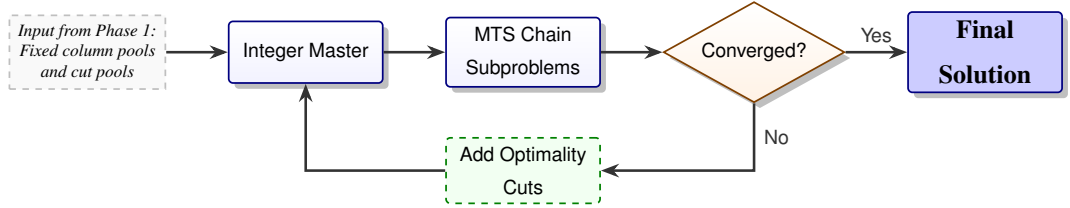


\subsection{Restricted Master Problem (RMP)}
\label{sec:RMP}
The Restricted Master Problem (RMP) solves the pattern-based formulation over restricted column pools $\bar{K}_b \subset K_b$ and $\bar{\Gamma}_{b,c,v} \subset \Gamma_{b,c,v}$. The supply-side constraints remain exact, while the downstream fulfillment cost is represented through Benders optimality cuts.

\begin{subequations} 
\label{form:RMP}
\allowdisplaybreaks
\begin{align}
  \min \quad 
  & \sum_{b \in B} \mathrm{cost}^{\text{run}}_b\,r_b
    + \sum_{b \in B}\sum_{c \in C}\sum_{v \in V} \mathrm{cost}^{\text{veh}}_{b,c,v} \sum_{\gamma \in \bar{\Gamma}_{b,c,v}} y^{\text{load}}_{b,c,v,\gamma} \nonumber \\
  & \quad + \sum_{p \in P} \sum_{c \in C^{\mathrm{MTO}}} \sum_{t \in T} \mathrm{cost}^{\text{lost}}_{p,c}\,l^{\mathrm{MTO}}_{p,c,t}
    + \sum_{p \in P} \sum_{c \in C^{\mathrm{MTS}}} \theta^{\mathrm{MTS}}_{p,c} \label{eq:RMP_obj} \\[6pt]
  \text{s.t.} \quad
  & \sum_{b \in B_{m,t}} r_b \le A_{m,t} 
    && \forall\, m \in M,\ t \in T \label{eq:RMP_cap} \\
  & r_b \le r^{\max}_b
    && \forall\, b \in B \label{eq:RMP_bound} \\
  & r_b \ge h_b \sum_{k \in \bar{K}_b} y^{\text{cut}}_{b,k} 
    && \forall\, b \in B \label{eq:RMP_prod} \\
  & \sum_{\gamma \in \bar{\Gamma}_{b,c,v}} a^{\text{load}}_{b,p,c,v,\gamma}\,y^{\text{load}}_{b,c,v,\gamma}
    \ge \sum_{k \in \bar{K}_b} a_{p,c,v,k}\,y^{\text{cut}}_{b,k} 
    && \forall\, b, p, c, v \label{eq:RMP_link} \\
  & \sum_{k \in \bar{K}_b} a_{p,c,v,k}\,y^{\text{cut}}_{b,k} \ge x^{\min}_{b,p,c,v}
    && \forall\, b, p, c, v \label{eq:RMP_min_ship} \\
  & S_{p,c,t} = \sum_{b \in B_t} \sum_{k \in \bar{K}_b} \sum_{v \in V} a_{p,c,v,k}\,y^{\text{cut}}_{b,k} 
    && \forall\, p, c, t \label{eq:RMP_supply_def} \\
  & S_{p,c,t} + l^{\mathrm{MTO}}_{p,c,t} \ge d_{p,c,t} 
    && \forall\, p,\ c \in C^{\mathrm{MTO}},\ t \label{eq:RMP_mto} \\
  & 0 \le l^{\mathrm{MTO}}_{p,c,t} \le d_{p,c,t}
    && \forall\, p,\ c \in C^{\mathrm{MTO}},\ t \label{eq:RMP_mto_lost_cap} \\
  & \theta^{\mathrm{MTS}}_{p,c} \ge 0
    && \forall\, p \in P,\ c \in C^{\mathrm{MTS}} \label{eq:RMP_theta_lb} \\
  & \theta^{\mathrm{MTS}}_{p,c} \ge \sigma^{h}_{p,c} + \sum_{t\in T} \phi^{h}_{p,c,t}\,S_{p,c,t}
    && \forall\, p \in P,\ c \in C^{\mathrm{MTS}},\ h \in \mathcal{H}_{p,c} \label{eq:RMP_benders} \\
  & y^{\text{cut}}_{b,k}, y^{\text{load}}_{b,c,v,\gamma} \in \mathbb{Z}_+. \nonumber
\end{align}
\end{subequations}

The objective \eqref{eq:RMP_obj} combines production cost, transportation cost, MTO lost-sales penalties, and the epigraph variables $\theta^{\mathrm{MTS}}_{p,c}$ that approximate downstream MTS cost. Here, $l^{\mathrm{MTO}}_{p,c,t}$ denotes lost sales for the MTO channel; the corresponding MTS lost-sales decisions are handled inside the chain subproblems.

Constraints \eqref{eq:RMP_cap}--\eqref{eq:RMP_link} carry forward the physical feasibility logic of the pattern-based model over the current column pools, and \eqref{eq:RMP_min_ship} enforces the minimum shipment obligations. Constraint \eqref{eq:RMP_supply_def} aggregates the generated quantities from the selected trimming patterns into the supply profile $S_{p,c,t}$. Constraints \eqref{eq:RMP_mto}--\eqref{eq:RMP_mto_lost_cap} retain the exact treatment of MTO lost sales from Section~\ref{sec:pattern_based_model}. Constraint \eqref{eq:RMP_theta_lb} provides a valid initial lower bound on downstream MTS cost, and \eqref{eq:RMP_benders} adds the current pool of Benders optimality cuts for each MTS chain $(p,c)$. Here, $\mathbf{S}_{p,c} = (S_{p,c,1}, \dots, S_{p,c,T})$ denotes the supply vector over the horizon; $\mathcal{H}_{p,c}$ indexes the cut pool; and each cut $h\in\mathcal{H}_{p,c}$ is specified by an intercept $\sigma^{h}_{p,c}$ and slope vector $(\phi^{h}_{p,c,t})_{t\in T}$. Together these cuts define a piecewise-linear lower approximation of the downstream LP chain value function $Z^{\mathrm{LP}}_{p,c}(\mathbf{S}_{p,c})$ described next.  In Phase~1, BDCG solves the LP relaxation of \eqref{form:RMP}; in Phase~2, integrality is restored on the retained column-usage variables over the same restricted column pools.
\subsection{Benders Subproblems}
\label{sec:benders_subproblem}

For a fixed $(p,c)$, let $Z^{\mathrm{LP}}_{p,c}(\mathbf{S}_{p,c})$ denote the minimum lost-sales and safety-stock penalty cost of the chain LP over the planning horizon. Let $Z^{\mathrm{IP}}_{p,c}(\mathbf{S}_{p,c})$ denote the value of the same chain problem under the original discrete inventory domains of $I^{+}_{p,c,t}$, $l_{p,c,t}$, and $I^{-}_{p,c,t}$. The LP value function is convex and piecewise linear in $\mathbf{S}_{p,c}$, which makes it amenable to a standard Benders outer approximation in the master. When $\mathbf{S}_{p,c}$ is integral, Lemma~\ref{lem:demand_integrality} below shows that $Z^{\mathrm{LP}}_{p,c}(\mathbf{S}_{p,c}) = Z^{\mathrm{IP}}_{p,c}(\mathbf{S}_{p,c})$.

\subsubsection{Primal Subproblem Formulation}

Fix $(p,c)\in P\times C^{\mathrm{MTS}}$ and treat $\mathbf{S}_{p,c}$ as fixed input data. The MTS subproblem determines the end-of-period inventory $I^{+}_{p,c,t}$, lost sales $l_{p,c,t}$, safety-stock shortfall $I^{-}_{p,c,t}$, and epigraph variables $\psi_{p,c,t}$ that represent the convex piecewise-linear safety-stock penalty. Associated dual multipliers are indicated in parentheses beside the constraints; only the flow-balance multipliers $\pi_{p,c,t}$ are used to construct the Benders cuts below.

Consistent with the exact formulation in Section~\ref{sec:pattern_based_model}, lost sales are bounded above by demand, with nonnegativity enforced by variable domain. This prevents the model from exploiting flow balance equations to artificially generate inventory by recording lost sales in excess of actual demand.
\begin{subequations}
\label{model:demand_primal}
\allowdisplaybreaks
\begin{align}
  Z^{\mathrm{LP}}_{p,c}(\mathbf{S}_{p,c}) = \min \quad
  & \sum_{t \in T} \left( \mathrm{cost}^{\text{lost}}_{p,c}\,l_{p,c,t} + \psi_{p,c,t} \right) \label{eq:demand_obj} \\
  \text{s.t.} \quad
  & I^{+}_{p,c,1} - l_{p,c,1} = I^{\text{init}}_{p,c} + S_{p,c,1} - d_{p,c,1}
    && (\pi_{p,c,1}) \label{eq:demand_flow_init} \\
  & I^{+}_{p,c,t} - I^{+}_{p,c,t-1} - l_{p,c,t} = S_{p,c,t} - d_{p,c,t}
    && \forall\, t > 1 \quad (\pi_{p,c,t}) \label{eq:demand_flow_rec} \\
  & I^{-}_{p,c,t} + I^{+}_{p,c,t} \ge ss_{p,c,t}
    && \forall\, t \in T \quad (\nu_{p,c,t}) \label{eq:demand_ss} \\
  & \psi_{p,c,t} \ge \alpha_{p,c,t,i}\,I^{-}_{p,c,t} - \beta_{p,c,t,i}
    && \forall\, t \in T,\ i \in N \quad (\omega_{p,c,t,i}) \label{eq:demand_cost} \\
  & 0 \le l_{p,c,t} \le d_{p,c,t}
    && \forall\, t \in T \label{eq:demand_lost_cap} \\
  & I^{+}_{p,c,t},\, I^{-}_{p,c,t},\, \psi_{p,c,t} \ge 0
    && \forall\, t \in T. \label{eq:demand_vars}
\end{align}
\end{subequations}


Note that for any nonnegative supply profile $\mathbf{S}_{p,c}$, the subproblem \eqref{model:demand_primal} is feasible and bounded. Feasibility follows because $l_{p,c,t}$ can always be chosen within $[0,d_{p,c,t}]$ and the shortfall variables $I^{-}_{p,c,t}$ can always satisfy the safety-stock constraints. Boundedness follows because all cost terms are nonnegative. Therefore, the demand-side subproblem never generates feasibility cuts; Benders \emph{optimality} cuts are sufficient.

\subsubsection{Construction of Benders Optimality Cuts}

Let $\hat{\mathbf{S}}_{p,c}$ be the supply profile produced by the current master solution, and let $\hat Z^{\mathrm{LP}}_{p,c}=Z^{\mathrm{LP}}_{p,c}(\hat{\mathbf{S}}_{p,c})$ be the corresponding optimal subproblem value. Let $\hat\pi_{p,c,t}$ denote an optimal dual multiplier associated with the flow-balance constraint at time $t$ (i.e., \eqref{eq:demand_flow_init} for $t=1$ and \eqref{eq:demand_flow_rec} for $t>1$).

Because the subproblem is a linear program in which $\mathbf{S}_{p,c}$ enters only through the right-hand side of the flow-balance constraints with coefficient $+1$, the vector $\hat\pi_{p,c}=(\hat\pi_{p,c,1},\dots,\hat\pi_{p,c,T})$ defines a subgradient of the LP value function $Z^{\mathrm{LP}}_{p,c}(\cdot)$ at $\hat{\mathbf{S}}_{p,c}$. The resulting supporting hyperplane yields the Benders optimality cut
\begin{equation}
\label{eq:benders_optimality_cut}
  \theta^{\mathrm{MTS}}_{p,c}
      \;\ge\;
  \hat Z^{\mathrm{LP}}_{p,c}
  \;+\;
  \sum_{t\in T} \hat\pi_{p,c,t}\,\big(S_{p,c,t}-\hat S_{p,c,t}\big),
  \qquad \forall\, p\in P,\ c\in C^{\mathrm{MTS}}.
\end{equation}
Cut \eqref{eq:benders_optimality_cut} is added to the restricted master problem as an additional lower bound on the epigraph variable $\theta^{\mathrm{MTS}}_{p,c}$, thereby tightening the master's approximation of the downstream LP value function.
\subsubsection{Integrality on Integral Supply Profiles}

When the supply profile is integral, the chain problem admits an equivalent incremental formulation with a totally unimodular constraint matrix. This argument uses the equivalent breakpoint representation of the convex penalty; by the discrete penalty-data convention stated in Section~\ref{sec:mathematical_formulation}, those breakpoints are integral.

\begin{lemma}[Integrality on Integral Supply Profiles]
\label{lem:demand_integrality}
Fix a customer-product pair $(p,c)$. Suppose the initial inventory $I^{\text{init}}_{p,c}$, demands $d_{p,c,t}$, safety-stock targets $ss_{p,c,t}$, and the convex piecewise-linear safety-stock penalty admit an equivalent breakpoint representation with integral breakpoints. Then, for any integral supply vector $\mathbf{S}_{p,c}$, subproblem \eqref{model:demand_primal} admits an equivalent incremental linear formulation with a totally unimodular constraint matrix and an integral right-hand side. Consequently, $Z^{\mathrm{LP}}_{p,c}(\mathbf{S}_{p,c}) = Z^{\mathrm{IP}}_{p,c}(\mathbf{S}_{p,c})$.
\end{lemma}

Written in cumulative incremental form, the chain has the consecutive-ones structure that yields total unimodularity. Lemma~\ref{lem:demand_integrality} therefore implies that integral supply profiles can be evaluated exactly by the LP subproblem, without explicit integrality restrictions on the chain variables. This is the setting relevant for Phase~2, because integral retained column-usage variables together with integral column coefficients induce integral supply profiles.

\subsubsection{Validity of Cuts Across Phases}

Each Benders cut \eqref{eq:benders_optimality_cut} is a globally valid supporting hyperplane for the convex LP value function $Z^{\mathrm{LP}}_{p,c}(\cdot)$. Consequently, cuts generated during Phase~1 (LP relaxation) remain valid when integrality is enforced in Phase~2. Since the integer feasible region is a subset of the continuous domain, any lower bound established for the continuous LP value function holds for the discrete points within it. On those integral Phase~2 profiles, Lemma~\ref{lem:demand_integrality} implies that $Z^{\mathrm{LP}}_{p,c}$ and $Z^{\mathrm{IP}}_{p,c}$ coincide. Reusing the accumulated cut pool therefore provides a strengthened initial approximation of the demand-side value function in the integer master, which can reduce the number of Phase~2 iterations.

\subsection{Pricing Subproblems}
\label{sec:pricing}

Column generation is applied to two supply-side column families: trimming patterns $K_b$ and lane-level loading configurations $\Gamma_{b,c,v}$. At each Phase~1 iteration, the LP relaxation of the RMP produces dual information that assigns an implicit value to additional supply-side columns.

For each MTS customer $c\in C^{\mathrm{MTS}}$ and product $p\in P$, the master includes the epigraph variable $\theta^{\mathrm{MTS}}_{p,c}$ and the associated supply profile $\mathbf{S}_{p,c}=(S_{p,c,t})_{t\in T}$. The current cut pool for $(p,c)$ is indexed by $h\in\mathcal{H}_{p,c}$ and consists of inequalities of the form
\[
  \theta^{\mathrm{MTS}}_{p,c} \;\ge\; \sigma^{h}_{p,c} + \sum_{t\in T}\phi^{h}_{p,c,t}\,S_{p,c,t}.
\]
Let $\xi^{h}_{p,c}\ge 0$ denote the dual multiplier of cut $h$ in the LP master. The combined marginal value of one additional unit of supply in period $t$ is therefore
\begin{equation}
  \kappa_{p,c,t}
  \;:=\;
  -\sum_{h\in\mathcal{H}_{p,c}} \xi^{h}_{p,c}\,\phi^{h}_{p,c,t},
  \qquad t\in T.
  \label{eq:price_from_cuts}
\end{equation}
This quantity is the channel through which the downstream LP chain value function influences Phase~1 pricing. Let $\rho_{p,c,t}$ denote the optimal dual multiplier of the supply-definition constraint \eqref{eq:RMP_supply_def}. Because $S_{p,c,t}$ is an unrestricted variable with zero objective coefficient that appears only in \eqref{eq:RMP_supply_def}, the MTO constraints \eqref{eq:RMP_mto}, and the cuts \eqref{eq:RMP_benders}, LP optimality forces its reduced cost to zero. Written out, this zero-reduced-cost condition is the dual identity
\begin{equation}
  \label{eq:supply_dual_identity}
  \rho_{p,c,t} \;=\; -\kappa_{p,c,t} \quad (c \in C^{\mathrm{MTS}}),
  \qquad
  \rho_{p,c,t} \;=\; -\zeta_{p,c,t} \quad (c \in C^{\mathrm{MTO}}),
\end{equation}
where $\zeta_{p,c,t}$ is the MTO fulfillment dual defined below: the supply dual transmits the downstream marginal value of one unit of period-$t$ supply, unchanged, to the supply side. Substituting \eqref{eq:supply_dual_identity} into the reduced cost of a trimming-pattern column yields the pricing conditions stated next.

\subsubsection{Trimming pricing problem (TPP; per run)}

Fix a run $b\in B_t$ for some period $t\in T$. Let $\mu_b$ be the dual multiplier associated with the run-time linking
constraint \eqref{eq:RMP_prod}, let $\lambda_{b,p,c,v}$ be the dual multiplier of the transportation coupling constraint
\eqref{eq:RMP_link}, and let $\eta_{b,p,c,v}\ge 0$ denote the dual multiplier of the minimum-shipment constraint
\eqref{eq:RMP_min_ship}. For MTO customers, let $\zeta_{p,c,t}\ge 0$ denote the dual multiplier of the MTO fulfillment
constraint \eqref{eq:RMP_mto}. For MTS customers, the MTS cut pool induces the per-period shipment credit $\kappa_{p,c,t}$ defined
in \eqref{eq:price_from_cuts}. Define the lane-specific net unit value
\begin{equation}
\label{eq:trim_net_value}
  \delta_{b,p,c,v}
  :=
  \begin{cases}
    \kappa_{p,c,t} - \lambda_{b,p,c,v} + \eta_{b,p,c,v}, & c \in C^{\mathrm{MTS}},\\
    \zeta_{p,c,t} - \lambda_{b,p,c,v} + \eta_{b,p,c,v}, & c \in C^{\mathrm{MTO}}.
  \end{cases}
\end{equation}
The TPP for run $b$ is
\begin{equation}
\label{eq:bdcg_trim_pricing}
  Z^{\text{TPP}}_b
  \;=\;
  \max \left\{
    \sum_{p \in P} \sum_{c \in C} \sum_{v \in V}
      \delta_{b,p,c,v}\,a_{p,c,v}
    \;:\;
    \begin{aligned}
      & \sum_{p \in P} \sum_{c \in C} \sum_{v \in V} w_p\,a_{p,c,v} \le W_b
    \end{aligned},
    \; a_{p,c,v} \in \mathbb{Z}_+
  \right\}.
    \end{equation}
An improving trimming pattern is found whenever $Z^{\text{TPP}}_b > h_b\,\mu_b$; an optimizer $a^\star$ yields a new pattern
column with coefficients $a_{p,c,v,k}=a^\star_{p,c,v}$.

\subsubsection{Loading pricing problem (LPP; per lane)}
The LPP uses the same coupling dual multipliers $\lambda_{b,p,c,v}$ from \eqref{eq:RMP_link} to generate an improving loading configuration:
    \begin{equation}
\label{eq:bdcg_load_pricing}
  Z^{\text{LPP}}_{b,c,v}
  \;=\;
  \max \left\{
    \sum_{p \in P} \lambda_{b,p,c,v}\,a^{\text{load}}_{p}
    \;:\;
    \begin{aligned}
      & \sum_{p \in P} g_{b,p}\,a^{\text{load}}_{p} \le G_v \\
      & \sum_{p \in P} q_p\,a^{\text{load}}_{p} \le Q_v
    \end{aligned},
    \; a^{\text{load}}_{p} \in \mathbb{Z}_+
  \right\}.
    \end{equation}
An improving loading configuration is found whenever $Z^{\text{LPP}}_{b,c,v} > \mathrm{cost}^{\text{veh}}_{b,c,v}$.

\subsubsection{Dynamic Programming for Trimming and Loading Pricing}
\label{sec:pricing_dp}

Both pricing problems are solved repeatedly during Phase~1 to identify columns with negative reduced cost. Once the master dual multipliers are fixed, trimming pricing and loading pricing reduce to exact integer knapsack models. Because all capacity coefficients in these pricing problems are measured in discrete units, as mentioned in Section~\ref{sec:mathematical_formulation}, the resulting knapsacks can be solved exactly by dynamic programming.

\textbf{Trimming pricing (1-dimensional knapsack)}
Consider the trimming pricing problem (TPP) for a fixed production run $b$,
defined in~\eqref{eq:bdcg_trim_pricing}.
TPP determines how a single jumbo roll should be cut and is subject to a
single capacity constraint: the deckle width $W_b$.
Each unit of product $p$ consumes width $w_p$, while contributing a net
unit reduced cost $\delta_{b,p,c,v}$ depending on the customer--lane pair
$(c,v)$.
TPP therefore has the structure of a \emph{one-dimensional integer knapsack
problem}.

While the net unit value $\delta_{b,p,c,v}$ depends on $(c,v)$, all units of
product $p$ consume the same width.
As a result, when determining how many units of each product to cut,
product $p$ can be represented by its best available net unit value
\[
  \bar\delta_{b,p} := \max_{c\in C,\,v\in V} \delta_{b,p,c,v}.
\]
\begin{lemma}[Productwise reduction of trimming pricing]
\label{lem:trim_productwise_reduction}
For a fixed run $b$, the optimal value of trimming pricing problem \eqref{eq:bdcg_trim_pricing} equals the optimal value of
\[
  \max \left\{
    \sum_{p\in P} \bar\delta_{b,p}\,x_p
    \;:\;
    \sum_{p\in P} w_p x_p \le W_b,\;
    x_p\in\mathbb{Z}_+
  \right\},
\]
Moreover, any optimal knapsack solution can be lifted to an optimal trimming pattern by assigning each product $p$ to any customer-lane pair attaining $\bar\delta_{b,p}$.
\end{lemma}
Lemma~\ref{lem:trim_productwise_reduction} reduces trimming pricing to a
standard one-dimensional integer knapsack problem.
This knapsack is solved exactly by a DP over scaled deckle width.
Once the optimal product counts $x_p$ are obtained, they are mapped back to
customer--lane pairs achieving $\bar\delta_{b,p}$.

\textbf{Loading pricing (two-dimensional knapsack)}
The loading pricing problem (LPP) for a fixed lane $(b,c,v)$, defined in
\eqref{eq:bdcg_load_pricing}, determines which products are loaded onto a
single vehicle.
Each unit of product $p$ consumes vehicle weight $g_{b,p}$ and loading
positions $q_p$, subject to capacity limits $G_v$ and $Q_v$.
LPP therefore has the structure of a \emph{bounded two-dimensional integer
knapsack problem}.

An exact DP over the two capacity dimensions is used to solve LPP.
Loading positions are handled via the same integer scaling as in trimming
pricing, while cumulative weight is tracked exactly in the DP state.

\subsubsection{Termination and LP optimality of Phase~1}
At termination, the current restricted master has already exhausted all sources of LP improvement. When neither pricing nor cut separation yields further improvement, Phase~1
terminates.

\begin{proposition}[LP optimality of Phase~1]
\label{prop:ph1_lp_optimality}
For each $(p,c)\in P\times C^{\mathrm{MTS}}$, let $\mathcal{T}_{p,c}$ denote the family of all supporting hyperplanes of $Z^{\mathrm{LP}}_{p,c}(\cdot)$. Suppose each Phase~1 restricted master LP is solved to optimality, the trimming and loading pricing subproblems are solved exactly, and demand-side separation returns a violated cut from $\mathcal{T}_{p,c}$ whenever $(\mathbf{S}_{p,c},\theta^{\mathrm{MTS}}_{p,c})$ violates the epigraph of $Z^{\mathrm{LP}}_{p,c}$. If Phase~1 terminates with no negative-reduced-cost trimming pattern, no negative-reduced-cost loading configuration, and no violated demand-side cut, then the current solution is optimal for the LP relaxation of the decomposed master over the full supply-side column families and the full cut families $\{\mathcal{T}_{p,c}\}$.
\end{proposition}

\subsection{Phase~2: Integer Benders refinement}
\label{sec:ph2_integer_refinement}

Phase~2 fixes the supply-side column pools generated in Phase~1 and enforces integrality on the retained column-usage variables. The method therefore reduces to an integer Benders refinement loop over a fixed set of trimming patterns, loading configurations, and accumulated cuts. No additional supply-side columns are generated in this phase. 
Because the retained column-usage variables are integral in this phase and all column coefficients are integral, the induced supply profiles are integral as well. By Lemma~\ref{lem:demand_integrality}, the LP evaluation of each inventory chain therefore coincides with its integer value $Z^{\mathrm{IP}}_{p,c}$, and downstream costs are evaluated exactly
throughout Phase~2.

Phase~2 is initialized from a feasible incumbent, followed by an initial pass
of demand-side cut separation in which violated Benders optimality cuts are added
until the incumbent satisfies all chain constraints.
This step improves the quality of the initial solution without altering the
Phase~2 master formulation, the set of retained columns, or the refinement logic.

\subsubsection{Refinement loop}
\label{sec:ph2_refine_loop}

Let $\ell=0,1,2,\dots$ index Phase~2 outer iterations. At iteration $\ell$, the integer master (RMP with fixed columns and the current cut pools) returns an incumbent solution $(y^\ell,\hat S^\ell,\hat\theta^\ell)$. For each $(p,c)\in P\times C^{\mathrm{MTS}}$, the corresponding inventory chain is evaluated at $\hat{\mathbf{S}}^\ell_{p,c}$.

\[
  \hat{\mathbf{S}}^{\ell}_{p,c}:=(\hat S^{\ell}_{p,c,t})_{t\in T}\in\mathbb{R}^{|T|}_+
  \qquad\text{and}\qquad
  \hat\theta^{\mathrm{MTS},\ell}_{p,c}.
\]
A chain is violated at iteration $\ell$ if
\begin{equation}
\label{eq:ph2_violation_test}
  Z^{\mathrm{LP}}_{p,c}(\hat{\mathbf{S}}^{\ell}_{p,c}) - \hat\theta^{\mathrm{MTS},\ell}_{p,c} \;>\; \varepsilon_{\mathrm{viol}}.
\end{equation}

Algorithm~\ref{alg:ph2_refinement} summarizes the fixed-pool refinement loop. Since pricing is disabled, the only constraints added in Phase~2 are demand-side chain cuts, and the supply-side column set remains fixed. In the computational study, this loop is terminated early if the remaining wall-clock budget is exhausted, in which case the best incumbent found so far is returned.

\begin{algorithm}[t]
\caption{Phase~2 integer Benders refinement (fixed columns)}
\label{alg:ph2_refinement}
\begin{algorithmic}[1]
\Require Fixed column pools from Phase~1; current cut pools $\{\mathcal{H}_{p,c}\}$; tolerance $\varepsilon_{\mathrm{viol}}$
\Repeat
  \State Solve the integer master (fixed columns, current cut pools) to obtain incumbent $(y,\hat{\mathbf{S}},\hat\theta^{\mathrm{MTS}})$
  \ForAll{$(p,c)\in P\times C^{\mathrm{MTS}}$}
    \State Solve chain LP \eqref{model:demand_primal} at $\mathbf{S}_{p,c}=\hat{\mathbf{S}}_{p,c}$ to obtain $Z^{\mathrm{LP}}_{p,c}(\hat{\mathbf{S}}_{p,c})$
    \If{$Z^{\mathrm{LP}}_{p,c}(\hat{\mathbf{S}}_{p,c})-\hat\theta^{\mathrm{MTS}}_{p,c}>\varepsilon_{\mathrm{viol}}$}
      \State Add one or more valid tangent cuts to $\mathcal{H}_{p,c}$
    \EndIf
  \EndFor
\Until{no chain violates \eqref{eq:ph2_violation_test}}
\State \Return incumbent $(y,\hat{\mathbf{S}},\hat\theta^{\mathrm{MTS}})$
\end{algorithmic}
\end{algorithm}


\begin{proposition}[Finite termination of Phase~2 refinement (fixed columns)]
\label{prop:ph2_finite_termination}
Assume Phase~2 is solved over fixed column pools and that each integer master problem is solved to optimality. Under Constraints~\eqref{eq:RMP_bound}, \eqref{eq:RMP_prod}, and the auxiliary finite-index conditions recorded in the Electronic Companion, the retained integer column-usage variables then range over a finite set. Assume also that every chain LP \eqref{model:demand_primal} is feasible and bounded for all $\mathbf{S}_{p,c}\ge 0$. If at each iteration $\ell$ at least one valid tangent Benders optimality cut is added for every violated chain, then the Phase~2 refinement loop terminates in finitely many iterations (up to tolerance $\varepsilon_{\mathrm{viol}}$).
\end{proposition}

\emph{Remark.} Proposition~\ref{prop:ph2_finite_termination} analyzes the refinement loop under the assumption that each integer master problem is solved to proven optimality before the next round of cuts is added. In the computational study of Section~\ref{sec:case_study}, each integer master problem is instead solved to a 2\% relative MIP-gap target and the loop stops when the wall-clock budget expires. Under this protocol, the returned incumbent remains integer-feasible and the Phase-1 root LP remains a valid lower bound, but optimality over the fixed column pools is not certified.
\section{Real-world Case Study}
\label{sec:case_study}

This section presents computational results that demonstrate the potential of the proposed BDCG-DP framework when evaluated on real-world industrial instances.

\subsection{Data Overview}
\label{subsec:case_instances}
The industrial benchmark evaluated in this study spans nine months and includes over 6,000 products and approximately 1,000 customers. The scale of these instances exceeds that of prior integrated studies in the paper industry, which consider smaller product sets, shorter planning horizons, or only partial integration of supply chain decisions \citep{CSPlotsizing2008,leao2017decomposition,pierini2021analysis,furlan2024matheuristic}. The real data model the full production-to-fulfillment pipeline over a rolling eight-week planning horizon. From each eight-week dataset, shorter planning-horizon instances are reconstructed to reflect standard industry practice and to enable a systematic evaluation of scalability of the proposed solution approach. Specifically,  one-week ($T_1$), two-week ($T_2$), four-week ($T_4$), and eight-week ($T_8$) instances are generated by restricting the decision horizon accordingly. This stratification yields a comprehensive test suite of 58 instances, consisting of 20 at $T_1$, 18 at $T_2$, 14 at $T_4$, and 6 at $T_8$. 

Table~\ref{tab:case_problem_size} summarizes the main dimensions of the problem across different horizon lengths. The dominant source of computational complexity is the number of feasible run--product--customer--vehicle ($b,p,c,v$) combinations. As the planning horizon increases, this combinatorial space grows  from 0.22M combinations at $T_1$ to 1.83M at $T_8$ (Table~\ref{tab:case_problem_size}), which increases the number of decision variables and constraints and makes the integrated problem challenging to solve.

\begin{table}[t]
  \centering
  \caption{Mean benchmark instance dimensions by horizon.}
  \label{tab:case_problem_size}
  {\scriptsize
  \setlength{\tabcolsep}{4.5pt}
  \renewcommand{\arraystretch}{1.05}
  \begin{tabular}{lrrr}
    \toprule
    \textbf{$T$ (n)} & \textbf{Production runs} & \textbf{Customer orders} & \textbf{Feasible run--product--customer--vehicle combinations}\\
    \midrule
    $T_1$ (20) & 528 & 8.9K & 0.22M \\
    $T_2$ (18) & 1079 & 17.9K & 0.45M \\
    $T_4$ (14) & 2200 & 36.2K & 0.92M \\
    $T_8$ (6)  & 4361 & 71.8K & 1.83M \\
    \bottomrule
  \end{tabular}
  }
\end{table}

\subsection{Solution Methods and Evaluation Criteria}
\label{subsec:case_setup}
This paper evaluates three solution methods: 
\begin{itemize}
    \item CG-MIP is a pure column generation where supply-side pricing subproblems are solved using a general-purpose MIP solver. Phase~1 corresponds to the LP column-generation stage, where the RMP and pricing subproblems are solved iteratively. Phase~2 consists of a final integer solve over the set of generated columns.
    \item CG-DP is a pure column generation where pricing subproblems are solved by dynamic programming (DP). Phase~1 and Phase~2 follow the same structure as in CG-MIP.
    \item BDCG-DP  uses the two-phase BDCG proposed in Section~\ref{sec:column_generation_bender} and solves pricing subproblems by DP. Phase~1 is the LP
decomposition stage, and Phase~2 is the integer Benders refinement over the retained columns.
\end{itemize}

A fourth configuration, BDCG-MIP, which pairs the BDCG architecture with MIP-based pricing, is evaluated in the Electronic Companion to separate the contribution of the pricing solver from that of the decomposition architecture.

To ensure a consistent basis for comparison across the three methods, solution quality is measured relative to a common theoretical lower bound $Z^{\mathrm{LP}}_{\mathrm{root}}$. This bound is obtained by solving the continuous LP relaxation of the full integrated model via column generation until reduced cost optimality is certified. Because this relaxation provides a valid lower bound for the original problem, it serves as a fixed, universal reference point for each instance. When evaluating the final integer solution produced by any of the three algorithms, the reported metric is $\mathrm{Gap}^{\mathrm{LP}} := \frac{Z^{\mathrm{IP}} - Z^{\mathrm{LP}}_{\mathrm{root}}}{\max\{1,|Z^{\mathrm{LP}}_{\mathrm{root}}|\}}\times 100\%$, where $Z^{\mathrm{IP}}$ is the final integer objective achieved by that algorithm. A lower $\mathrm{Gap}^{\mathrm{LP}}$ therefore indicates a solution closer to the theoretical lower bound. For every instance, the column-generation procedure certifies reduced-cost optimality of this relaxation within the Phase-1 budget as shown in Appendix Table~\ref{tab:ec_four_method_phase1}.

\subsection{Implementation Environment}

All codes are implemented in Python~3.10.16 and solved with Gurobi~12.0.1. The benchmark runs are executed
as single-task Slurm jobs on matched Gold~6226 compute nodes with 24 CPUs allocated per task.
All methods start from the same
common initial trimming-pattern set and face the same total 8-hour wall-clock budget. For all methods, Phase~1 is capped at
4 hours, and any unused Phase~1 time is inherited by Phase~2 or the final integer solve.  All
final integer solves use the same 2\% terminal MIP-gap target. The Electronic Companion reports the common initial-pattern construction, supplementary comparison tables, and the Phase~2 warm-start procedure.

\subsection{Computational Results}
\label{subsec:case_results}


\subsubsection{Short Planning Horizon Results}
At $T_1$ and $T_2$, the proposed BDCG-DP framework delivers median runtime reductions by factors of 16.3 and 8.4 relative to the standard CG-MIP baseline. Table~\ref{tab:case_summary_metrics} shows that simply replacing generic MIP pricing with exact DP pricing, as demonstrated by the CG-DP method, drops the median end-to-end runtime from 30.5 to 2.2 minutes at $T_1$ and from 51.0 to 6.5 minutes at $T_2$. Because BDCG-DP incorporates this same exact pricing mechanism, it achieves similar accelerations in these short-horizon settings. Specifically, BDCG-DP runs 16.3 times faster than CG-MIP at $T_1$ and 8.4 times faster at $T_2$. The final objective statistics in Table~\ref{tab:case_summary_metrics} further show that these runtime improvements do not come at the expense of solution quality. Across all three methods, final objective values are nearly identical. For these shorter planning windows, exact pricing is the primary engine of acceleration, allowing the framework to run to completion within minutes.

\begin{table*}[t]
  \centering
  \caption{Median end-to-end runtime, final objective, and root-LP gap by horizon.}
  \label{tab:case_summary_metrics}
  {\scriptsize
  \setlength{\tabcolsep}{4.0pt}
  \renewcommand{\arraystretch}{1.08}

  \textbf{End-to-end runtime (minutes)}

  \vspace{0.3em}

  \begin{tabular*}{\textwidth}{@{\extracolsep{\fill}}lccccc}
    \toprule
    \textbf{$T$ (n)} & \textbf{CG-MIP} & \textbf{CG-DP} & \textbf{BDCG-DP} & \shortstack{\textbf{CG-DP}\\\textbf{speedup}} & \shortstack{\textbf{BDCG-DP}\\\textbf{speedup}} \\
    \midrule
    $T_1$ (20) & 30.52 & 2.19 & \textbf{1.87} & 14.0x & 16.3x \\
    $T_2$ (18) & 50.97 & 6.55 & \textbf{6.04} & 7.8x & 8.4x \\
    $T_4$ (14) & 341.35 & 328.20 & \textbf{42.34} & 1.0x & 8.1x \\
    $T_8$ (6) & 480 (Cap) & 480 (Cap) & \textbf{480 (Cap)} & \add{--} &\add{--} \\
    \bottomrule
  \end{tabular*}

  \vspace{0.9em}

  \textbf{Final objective and root-LP gap}

  \vspace{0.3em}

  \begin{tabular*}{\textwidth}{@{\extracolsep{\fill}}lccc c ccc c}
    \toprule
    & \multicolumn{3}{c}{\textbf{Median objective ($\times 10^9$)}} & \multicolumn{1}{c}{\textbf{Improvement}} & \multicolumn{3}{c}{\textbf{Median gap$^{\mathrm{LP}}$ (\%)}} & \multicolumn{1}{c}{\textbf{Improvement}} \\
    \cmidrule(lr){2-4}\cmidrule(lr){5-5}\cmidrule(lr){6-8}\cmidrule(lr){9-9}
    \textbf{$T$ (n)} & \textbf{CG-MIP} & \textbf{CG-DP} & \textbf{BDCG-DP} & \shortstack{\textbf{Obj. reduction}\\\textbf{vs CG-MIP}} & \textbf{CG-MIP} & \textbf{CG-DP} & \textbf{BDCG-DP} & \shortstack{\textbf{Gap ratio}\\\textbf{CG-MIP / BDCG-DP}} \\
    \midrule
    $T_1$ (20) & 0.36 & 0.36 & \textbf{0.36} & 0.5\% & 3.14 & 3.04 & \textbf{2.30} & 1.36x \\
    $T_2$ (18) & 0.60 & 0.59 & \textbf{0.59} & 0.4\% & 3.61 & 3.24 & \textbf{3.11} & 1.16x \\
    $T_4$ (14) & 1.04 & 1.04 & \textbf{1.04} & 0.5\% & 3.71 & 3.40 & \textbf{3.27} & 1.13x \\
    $T_8$ (6) & 2.72 & 2.71 & \textbf{2.05} & 24.7\% & 28.11 & 28.43 & \textbf{4.12} & 6.82x \\
    \bottomrule
  \end{tabular*}

  \vspace{0.3em}
  {\par\raggedright\scriptsize \textit{Notes:} $n$ $=$ number of instances per horizon; Cap $=$ the 480-minute end-to-end budget is reached; the speedup columns report ratios of median runtimes relative to CG-MIP. At $T_8$ the median runtime of every method equals the budget, so runtime ratios are omitted. $\mathrm{Gap}^{\mathrm{LP}}$ is defined in Section~\ref{subsec:case_setup}.\par}
  }
\end{table*}

\subsubsection{Long Planning Horizon Results}

As the planning horizon expands to four and eight weeks, the structural advantages of the BDCG architecture become the primary driver of performance. At $T_4$, the proposed BDCG framework continues to deliver a factor-of-eight speedup. Table~\ref{tab:case_summary_metrics} shows that BDCG-DP achieves a median end-to-end runtime of 42.3 minutes, which is approximately a factor of eight faster than the 341.4 minutes required by CG-MIP and the 328.2 minutes required by CG-DP. At $T_8$, the CG baselines exhaust the common 8-hour budget on all six instances. BDCG-DP exhausts the budget on four of the six instances, so its median total runtime also equals the cap, while its fastest run finishes after 307.7 minutes (Table~\ref{tab:ec_four_method_endtoend}). In this time-constrained setting, the evaluation shifts from runtime to final solution quality. BDCG-DP lowers the median final objective from 2.71 billion under CG-DP to 2.05 billion, a 24.4\% reduction in the median objective value. Simultaneously, the median gap to the root-LP bound decreases from approximately 28\% under the CG methods down to just 4.12\%.

\begin{figure*}[t]
 \centering
  \includegraphics[width=0.86\textwidth]{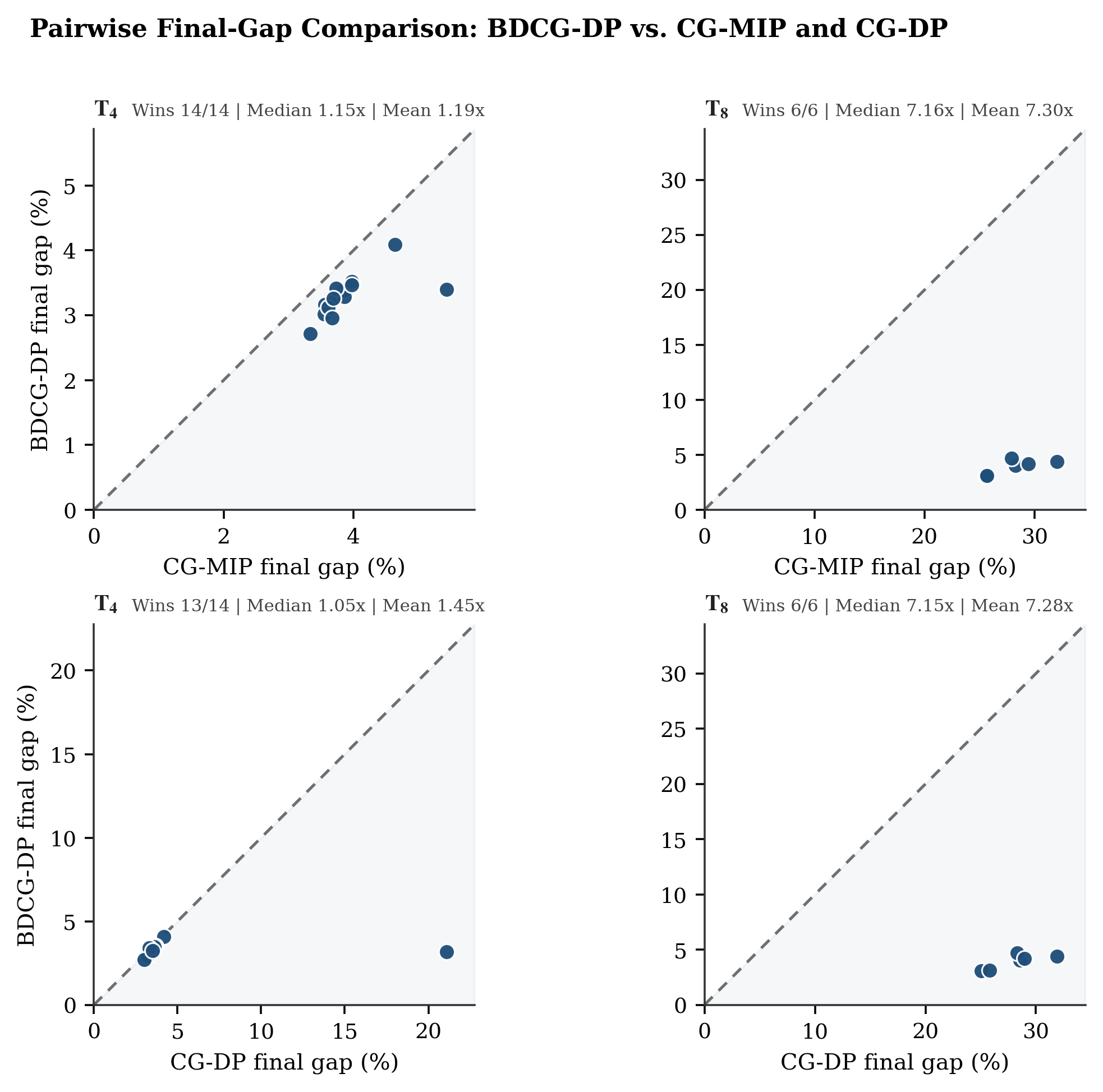}
  \caption{Instance-level comparison of final gap to the common root-LP bound on the $T_4$ and $T_8$ instances. The
  top row compares BDCG-DP against CG-MIP, and the bottom row compares BDCG-DP against CG-DP. Points below the
  45-degree line favor BDCG-DP; panel headers report win counts and summary gap ratios.}
  \label{fig:case_pairwise_dominance}
\end{figure*}

Figure~\ref{fig:case_pairwise_dominance} provides an instance-by-instance comparison. At $T_4$, BDCG-DP attains a lower final gap on 13 of the 14 instances against both the CG-MIP and CG-DP baselines. As the problem scales to $T_8$, the framework attains a lower final root-LP gap on every instance against both methods. This consistent performance confirms that decoupling the supply-side combinatorics from the downstream inventory problem is essential for scaling exact integrated planning to industrial eight-week horizons.

\subsection{Analyzing the Added Advantage of Different Decomposition Components}

While the computational results establish that BDCG-DP attains the best median runtime on the one- to four-week horizons and the best median root-LP gap at every horizon, one may wonder why decoupling the supply-side combinatorics from the downstream inventory evaluation improves performance as the problem scales. To answer this, it is necessary to examine how the computational burden shifts between the continuous relaxation and the integer refinement stages, what that shift buys in terms of feasible incumbents, and which cost components account for the final objective improvement.

\textbf{Shifting the computational burden.}
Figure~\ref{fig:case_phase_split} illustrates how BDCG-DP redistributes the computational workload. As the planning horizon expands, the integrated CG baselines spend an increasing majority of their time in Phase~2 integer refinement. BDCG-DP mitigates this by dedicating more time to Phase~1, where it simultaneously prices columns and generates Benders cuts. This additional effort in the continuous relaxation reduces the downstream integer-refinement time. At $T_4$, the median Phase~2 solve time for BDCG-DP is 30.5 minutes, compared to over 260 minutes for the integrated baselines. Even at $T_8$, where all methods reach the time limit, BDCG-DP enters Phase~2 after completing more first-stage work (median Phase-1 time of 133.59 versus 46.17 minutes for CG-DP; Table~\ref{tab:ec_four_method_phase1}). Additional Phase~1 diagnostics are reported in the Electronic Companion.

\begin{figure*}[t]


  \includegraphics[width=0.86\textwidth]{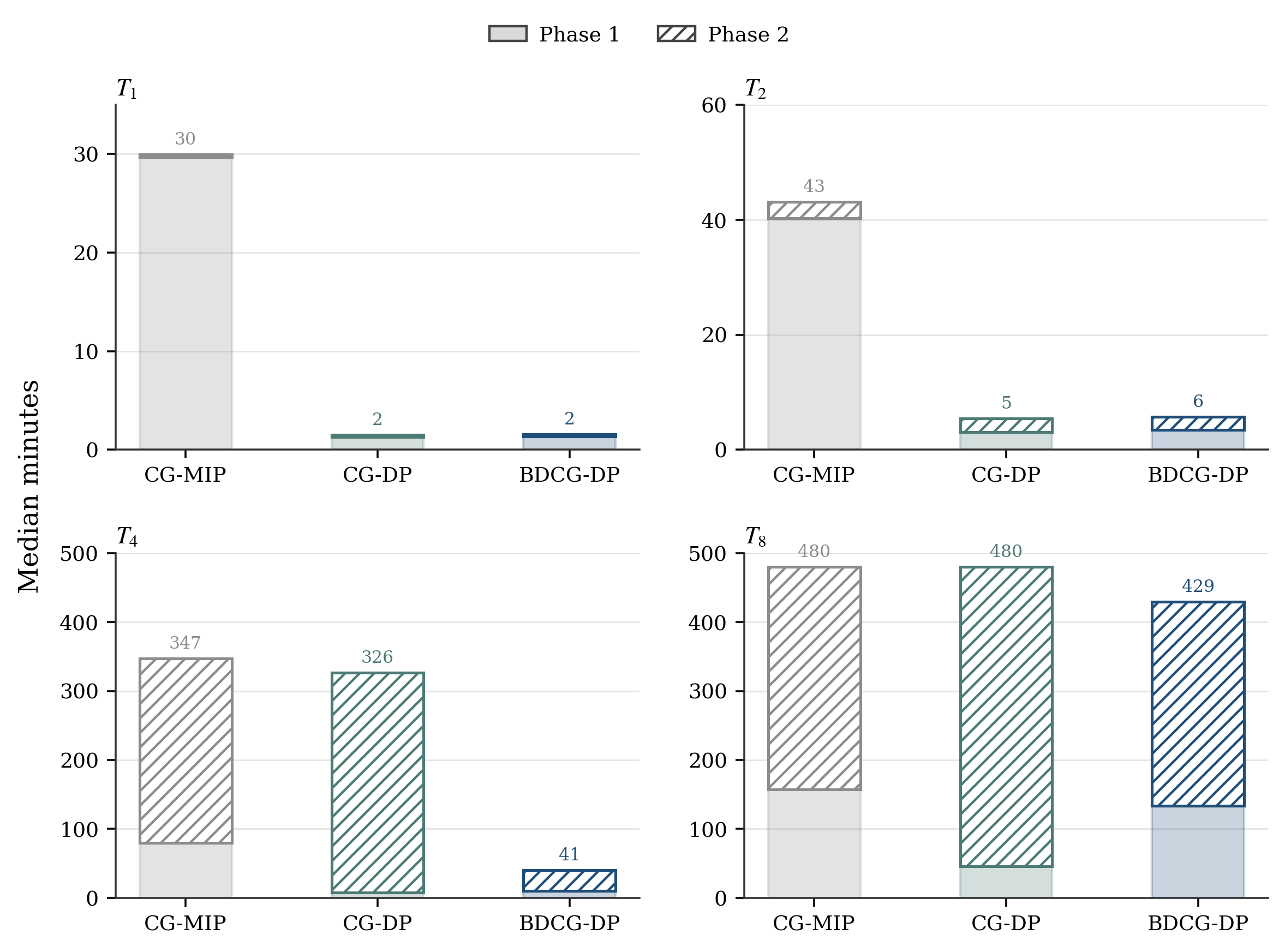}
 \captionsetup{type=figure,skip=4pt}
  \caption{Median phase-split runtime by horizon for CG-MIP, CG-DP, and BDCG-DP. Shaded segments report Phase~1 time,
  hatched segments report Phase~2 time, and labels above the bars report the sum of the displayed phase medians.}
 \label{fig:case_phase_split}
\end{figure*}

\textbf{Earlier and improved incumbents.}
Figure~\ref{fig:case_phase2_progress} shows the practical effect of this Phase~1 effort. On the six $T_8$ instances, BDCG-DP reaches its first integer-feasible solution between 122 and 343 minutes from the algorithm start. Each of these \emph{initial} BDCG-DP incumbents achieves a lower objective value than the \emph{final} end-of-budget incumbent produced by CG-DP. This pattern suggests that the additional Phase~1 work provides Phase~2 with a tighter starting formulation, allowing feasible solutions to emerge earlier and with lower costs.

\begin{figure*}[t]
  \centering
  \includegraphics[width=0.92\textwidth]{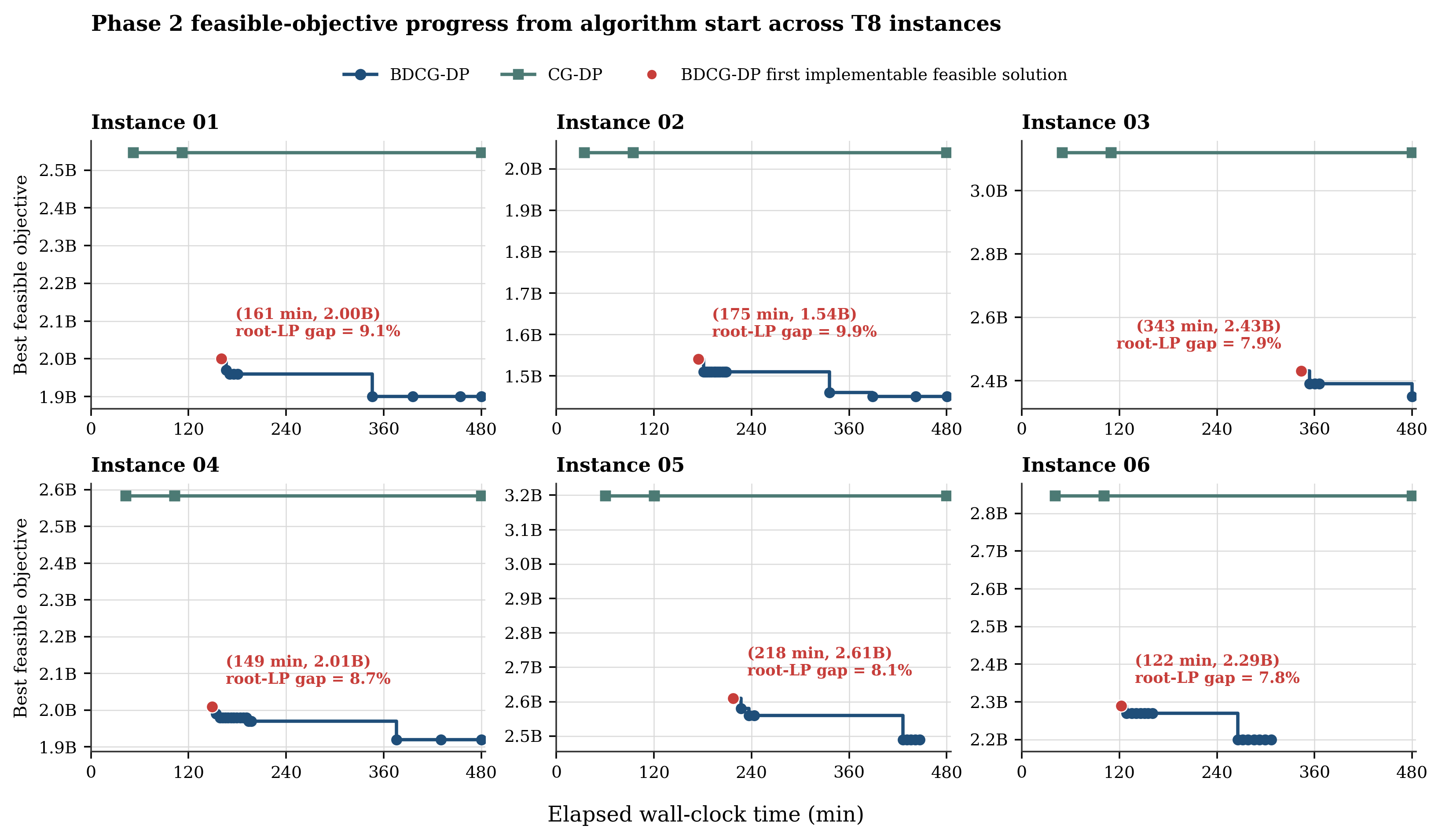}
  \caption{End-to-end feasible-objective progress on the $T_8$ instances. Each curve shows the best implementable objective found so far; red markers denote the first feasible BDCG-DP solution, and the adjacent label reports the corresponding root-LP gap.}
  \label{fig:case_phase2_progress}
\end{figure*}

\textbf{MTS penalties drive the cost reduction.}
Figure~\ref{fig:case_cost_components} breaks down the final objective into its core cost components. At $T_8$, the median production and transportation costs remain stable when moving from CG-DP to BDCG-DP. The improvement occurs instead within the downstream fulfillment portion of the objective. BDCG-DP lowers the median MTS penalty cost from 1.68 billion to 1.03 billion, which accounts for the reduction in the total objective. This aligns with the structure of the decomposition, as the demand-side Benders cuts directly target the downstream MTS value.

\begin{figure*}[t]
  \centering
  \includegraphics[width=0.94\textwidth]{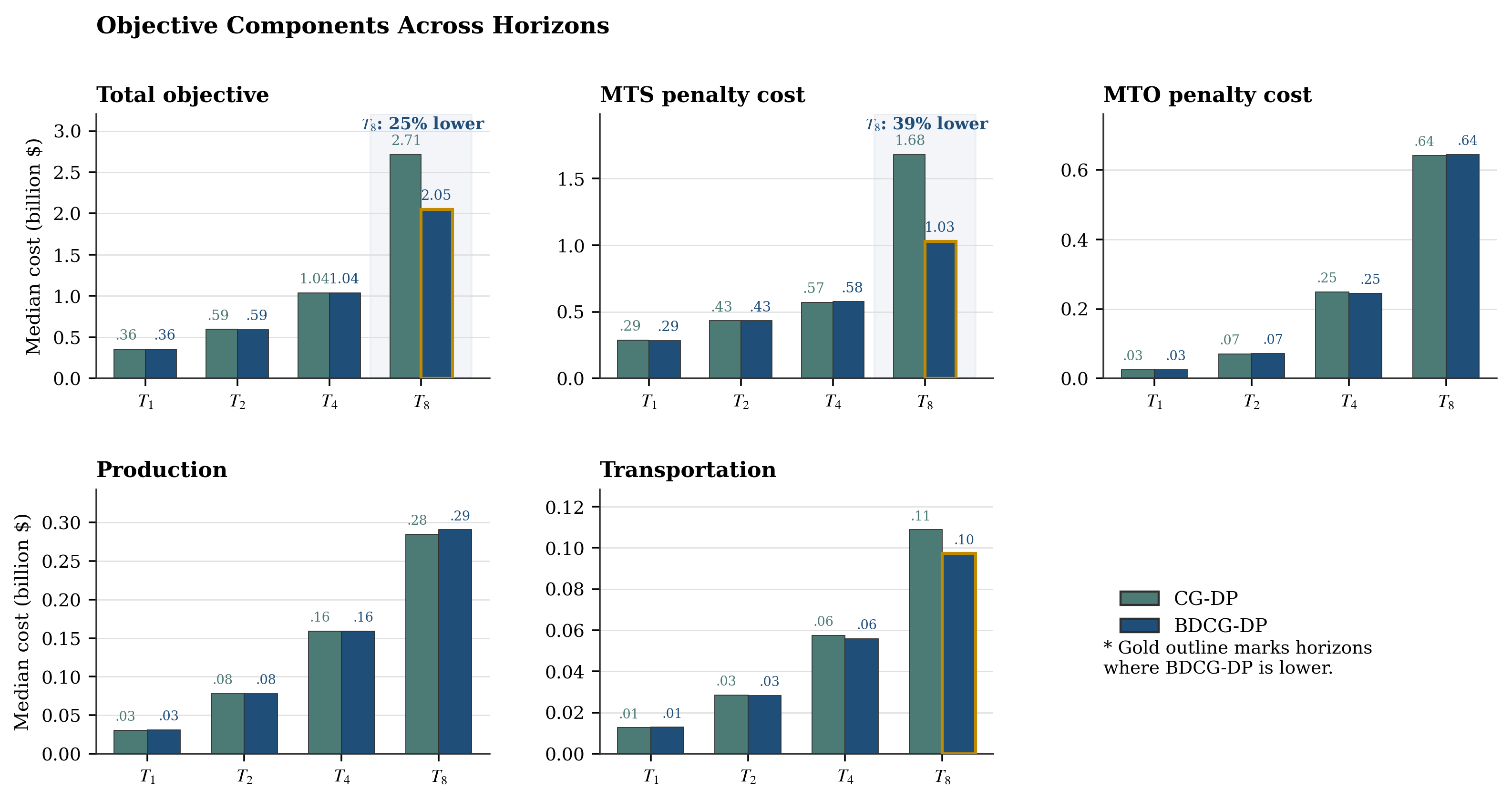}
  \caption{Median objective decomposition for CG-DP and BDCG-DP by horizon.}
  \label{fig:case_cost_components}
\end{figure*}

\subsection{Managerial Insights}
For decision-makers, the practical value of a planning tool depends on both the final solution quality and how quickly a usable plan becomes available. Figure~\ref{fig:case_phase2_progress} shows that the first feasible BDCG-DP plan arrives within 2.3 to 6.0 hours before the common time budget expires and already improves upon the final incumbent returned by CG-DP. The method therefore offers two practical benefits. It provides a lower-cost final schedule, and it delivers an integer-feasible operational plan hours before the time budget is exhausted. In a rolling-horizon planning environment, where decisions are re-optimized periodically as new information becomes available, early access to a lower-cost plan can be as important as the final end-of-budget improvement itself.


\section{Conclusion}
\label{sec:conclusion}

This paper introduces an exact integrated model for the end-to-end paper manufacturing and distribution supply chain, jointly optimizing production, cutting stock, vehicle-feasible loading, and multi-period fulfillment. The central methodological insight is that upstream production and logistics decisions interact with downstream fulfillment only through the exact quantities of each product delivered to each customer in each period. This structural separation yields the BDCG-DP framework. The approach uses column generation with exact dynamic-programming pricing to manage the exponentially large supply-side pattern families, while simultaneously using Benders decomposition to tighten the root LP relaxation with respect to downstream fulfillment costs.

The computational evidence on 58 proprietary industrial instances demonstrates the scalability of this architecture. On short planning horizons, the framework accelerates end-to-end solution times by factors of 16.3 (one-week) and 8.4 (two-week) relative to the CG-MIP baseline, an acceleration driven by the exact dynamic-programming pricing. On the eight-week horizons, where the benchmarked methods face a common eight-hour budget, the proposed BDCG-DP achieves a lower final root-LP gap on all six instances against both baselines. By entering integer refinement with a compact master and a pre-generated pool of demand-side cuts, it reduces the median gap from 28.43\% to 4.12\% and lowers final objective values by 24.4\% relative to CG-DP. For operational decision-makers, the framework also delivers integer-feasible plans with root-LP gaps below 10\% in 2.3 to 6.0 hours before the computational budget expires.


%
%
%







\medskip\noindent\textbf{Data Ethics \& Reproducibility Note.} The 58 benchmark instances were provided by an industrial partner under a confidentiality agreement and cannot be shared. The Electronic Companion documents the initial-pattern construction, the Phase~2 warm-start procedure, and the experimental protocol used by all four benchmark methods.

\bibliographystyle{informs2014}
\bibliography{sample}


\ECSwitch
\ECHead{Electronic Companion: Additional Computational Evidence and Proofs}
\renewcommand{\theequation}{\thesection.\arabic{equation}}

\section{Additional computational evidence}
\label{sec:ec_computational}

\subsection{Initial trimming-pattern set}
\label{sec:ec_initial_patterns}
Column generation requires an initial restricted master that contains enough columns to satisfy the mandatory shipment constraints \eqref{eq:RMP_min_ship}; otherwise, the Phase-1 restricted LP is infeasible. For each benchmark instance, a common initial trimming-pattern pool is therefore constructed before any of the compared algorithms begins. This preprocessing uses only instance data; it does not use dual prices, Benders cuts, or information from any benchmark method. All notation is as defined in Section~\ref{sec:mathematical_formulation}; the only auxiliary symbol is the index set $\mathcal{A}_b$ introduced below, which is used in this appendix alone.

The construction has three steps. First, a mandatory-coverage pool is built. For each run $b$, let
\[
\mathcal{A}_b := \bigl\{(p,c,v) : x^{\min}_{b,p,c,v} > 0\bigr\}
\]
denote the product--customer--vehicle combinations with positive minimum-shipment requirements from that run. For every run with $\mathcal{A}_b \neq \emptyset$, the following width-feasibility covering problem assigns the required units to the candidate jumbo-roll slots $o \in O_b$ of Section~\ref{sec:item_based_model}, reusing the item-based cutting variables $x_{o,p,c,v}$ and roll indicators $u_{b,o}$:
\begin{equation}
\label{eq:ec_initial_cover}
\begin{aligned}
  \min \quad & \sum_{o \in O_b} u_{b,o} \\
  \text{s.t.} \quad
  & \sum_{o \in O_b} x_{o,p,c,v} \ge x^{\min}_{b,p,c,v}
    && \forall\, (p,c,v) \in \mathcal{A}_b, \\
  & \sum_{(p,c,v) \in \mathcal{A}_b} w_p\, x_{o,p,c,v} \le W_b\, u_{b,o}
    && \forall\, o \in O_b, \\
  & x_{o,p,c,v} \in \mathbb{Z}_+, \quad u_{b,o} \in \{0,1\}.
\end{aligned}
\end{equation}
Each used slot $o$ (i.e., $u_{b,o} = 1$) defines one width-feasible trimming pattern $k$ with coefficients $a_{p,c,v,k} = x_{o,p,c,v}$ for $(p,c,v) \in \mathcal{A}_b$ and $a_{p,c,v,k} = 0$ otherwise, and the minimized slot count $\sum_{o \in O_b} u_{b,o}$ provides a lower bound on the number of jumbo rolls required in run $b$. Any feasible solution of the item-based model induces, by restriction of its cutting decisions to the tuples in $\mathcal{A}_b$, a feasible solution of \eqref{eq:ec_initial_cover}, so the covering problem is feasible on every benchmark instance. This step ensures that the initial pool contains patterns covering all positive minimum-shipment rows \eqref{eq:RMP_min_ship}.

Second, a compact aggregate planning model is solved, with the run counts from the mandatory-coverage step enforced as lower bounds. This model uses aggregate run-count and shipment variables rather than enumerated trimming patterns, and includes the production, transportation, and demand and inventory logic of Section~\ref{sec:mathematical_formulation}. Its solution serves only as a guide: it provides target run counts and positive shipment quantities for the diversification step. Its objective value is not used in the computational comparison.

Third, the remaining pattern budget is filled with randomized width-feasible patterns. Sampling effort is allocated across runs according to the aggregate run counts and active shipment lanes from the second step. Each sampled pattern selects a random subset of product--lane combinations and assigns integer quantities within the deckle-width limit; empty or width-infeasible candidates are rejected, and duplicates are removed.

The final initial trimming-pattern pool is the deduplicated union of the mandatory-coverage patterns and the randomized diversification patterns, capped at 30{,}000 distinct patterns per instance. The same pool is supplied to CG-MIP, CG-DP, BDCG-MIP, and BDCG-DP, and its construction time is identical across methods. Differences in runtime and final gaps are therefore not attributable to method-specific initial trimming columns.
\subsection{Initial loading configurations}
\label{sec:ec_initial_load_cfgs}

The restricted master needs initial loading configurations for the same reason it needs initial trimming patterns. The coupling constraints \eqref{eq:RMP_link} require every roll cut for a lane $(b,c,v)$ to be covered by vehicle capacity reserved on that lane, and vehicle capacity can only be reserved by selecting loading-configuration columns from the pool $\bar{\Gamma}_{b,c,v}$. Because the minimum-shipment rows \eqref{eq:RMP_min_ship} force positive cut quantities, an initial pool without loading columns on the affected lanes would make the very first restricted LP infeasible.

The initial loading pool is constructed by a simple rule at master-assembly time, not enumerated offline. For every lane $(b,c,v)$ and every product $p$ that some initial trimming pattern $k$ ships on that lane (i.e., $a_{p,c,v,k}>0$), the pool receives one \emph{single-product} configuration $\gamma \in \Gamma_{b,c,v}$ that loads
\begin{equation}
\label{eq:ec_seed_load_cfg}
  q^{\max}_{b,p,v}
  \;:=\;
  \min\left\{
    \left\lfloor G_v / g_{b,p} \right\rfloor,\;
    \left\lfloor Q_v / q_p \right\rfloor
  \right\}
\end{equation}
rolls of product $p$ and nothing else, where $G_v$ and $Q_v$ are the per-vehicle weight and position capacities and $g_{b,p}$ and $q_p$ are the unit weight and position requirements of Section~\ref{sec:mathematical_formulation}. By construction, $\gamma$ is the largest load of product $p$ alone that a single vehicle on lane $(b,c,v)$ can carry, so it is a feasible element of $\Gamma_{b,c,v}$.

This seeding is sufficient for feasibility. In the Phase-1 LP relaxation the configuration-usage variables $y^{\text{load}}$ are continuous, so any cut quantity $q$ of product $p$ on lane $(b,c,v)$ can be covered by selecting the seeded configuration at level $q / q^{\max}_{b,p,v}$; in Phase~2, integer multiples of the seeded configurations play the same role. We then generate the initial patterns that mix products on one vehicle during Phase~1 by the loading pricing problem \eqref{eq:bdcg_load_pricing}. The rule uses only instance data, is identical across CG-MIP, CG-DP, BDCG-MIP, and BDCG-DP, and its (negligible) construction time is incurred inside each method's measured runtime. In contrast, the shared trimming-pattern pool of Section~\ref{sec:ec_initial_patterns} is constructed once per instance in preprocessing, outside the 8-hour budget, and reused by all methods.

\subsection{Interpreting the phase-split comparisons}
\label{sec:ec_phase_split_interpretation}

When comparing phase splits across the different architectures, the raw Phase~1 times are not directly equivalent. In the integrated CG-MIP and CG-DP methods, Phase~1 is a pure column-generation stage over the integrated master. In the decomposed BDCG-MIP and BDCG-DP methods, Phase~1 simultaneously generates supply-side columns and demand-side Benders cuts. A BDCG method naturally spends more time in Phase~1 because it performs more work to close the downstream approximation before integer refinement begins. The relevant metric is therefore not which method has the shortest Phase~1, but how that Phase~1 investment improves the final feasible solution quality achieved under the common end-to-end budget.

\subsection{Full four-method timing comparison}
\label{sec:ec_four_method}

The four-method ablation tables provide a detailed view of how the pricing solver and the decomposition architecture interact.

Tables~\ref{tab:ec_four_method_endtoend}--\ref{tab:ec_four_method_phase2} report the median, mean, and cross-instance range for all four benchmark configurations. The end-to-end table reports the final root-LP gap, while the Phase~1 and Phase~2 tables isolate the wall-clock time before and after the integer-refinement handoff.

\begin{table*}[t]
    \centering
    \caption{End-to-end wall-clock time (minutes) and final root-LP gap (\%) for CG-MIP, CG-DP, BDCG-MIP, and BDCG-DP. `Cap' denotes total time at or above the 480-minute budget.}
    \label{tab:ec_four_method_endtoend}
    {\scriptsize
    \setlength{\tabcolsep}{2.8pt}
    \renewcommand{\arraystretch}{1.06}
    \resizebox{\textwidth}{!}{%
    \begin{tabular}{l rrr rrr rrr rrr}
      \toprule
      & \multicolumn{3}{c}{\textbf{CG-MIP}} & \multicolumn{3}{c}{\textbf{CG-DP}} & \multicolumn{3}{c}{\textbf{BDCG-MIP}} & \multicolumn{3}{c}{\textbf{BDCG-DP}} \\
      \cmidrule(lr){2-4}\cmidrule(lr){5-7}\cmidrule(lr){8-10}\cmidrule(lr){11-13}
      \textbf{$T$ (n)} & Median & Mean & {[Min, Max]} & Median & Mean & {[Min, Max]} & Median & Mean & {[Min, Max]} & Median & Mean & {[Min, Max]} \\
      \midrule
      \multicolumn{13}{l}{\emph{Total Runtime (min)}} \\
      \midrule
      $T_1$ (20) & 30.52 & 30.55 & {[4.65, 81.46]} & 2.19 & 2.38 & {[1.24, 4.35]} & 21.23 & 32.29 & {[4.29, 87.57]} & \textbf{1.87} & \textbf{2.08} & \textbf{{[1.18, 3.32]}} \\
      $T_2$ (18) & 50.97 & 51.32 & {[10.43, 131.62]} & 6.55 & 25.67 & {[4.49, 138.86]} & 47.55 & 51.40 & {[16.23, 88.61]} & \textbf{6.04} & \textbf{6.63} & \textbf{{[3.58, 16.95]}} \\
      $T_4$ (14) & 341.35 & 344.83 & {[203.35, Cap]} & 328.20 & 331.86 & {[207.34, Cap]} & 140.68 & 142.11 & {[85.48, 231.23]} & \textbf{42.34} & \textbf{44.13} & \textbf{{[18.36, 85.43]}} \\
      $T_8$ (6) & Cap & Cap & {Cap} & Cap & Cap & {Cap} & Cap & 472.57 & {[430.03, Cap]} & \textbf{Cap} & \textbf{446.67} & \textbf{{[307.65, Cap]}} \\
      \midrule
      \multicolumn{13}{l}{\emph{Final Gap$^{\mathrm{LP}}$ (\%)}} \\
      \midrule
      \textbf{$T$ (n)} & Median & Mean & {[Min, Max]} & Median & Mean & {[Min, Max]} & Median & Mean & {[Min, Max]} & Median & Mean & {[Min, Max]} \\
      \midrule
      $T_1$ (20) & 3.14 & 3.18 & {[2.18, 4.11]} & 3.04 & 2.97 & {[2.31, 3.73]} & 2.85 & 2.83 & {[1.45, 4.15]} & \textbf{2.30} & \textbf{2.37} & \textbf{{[1.13, 3.38]}} \\
      $T_2$ (18) & 3.61 & 3.57 & {[3.17, 4.11]} & 3.24 & 3.26 & {[2.85, 3.87]} & 3.49 & 3.53 & {[2.82, 4.18]} & \textbf{3.11} & \textbf{3.09} & \textbf{{[2.29, 3.91]}} \\
      $T_4$ (14) & 3.71 & 3.90 & {[3.33, 5.43]} & 3.40 & 4.70 & {[3.00, 21.09]} & 3.72 & 3.74 & {[3.38, 4.63]} & \textbf{3.27} & \textbf{3.28} & \textbf{{[2.72, 4.09]}} \\
      $T_8$ (6) & 28.11 & 28.16 & {[25.55, 32.06]} & 28.43 & 28.09 & {[25.03, 31.92]} & 4.40 & 4.70 & {[3.49, 7.50]} & \textbf{4.12} & \textbf{3.92} & \textbf{{[3.07, 4.68]}} \\
      \bottomrule
    \end{tabular}%
    }
    }
  \end{table*}

\begin{table*}[t]
    \centering
    \caption{Phase~1 wall-clock time (minutes) for CG-MIP, CG-DP, BDCG-MIP, and BDCG-DP.}
    \label{tab:ec_four_method_phase1}
    {\scriptsize
    \setlength{\tabcolsep}{2.8pt}
    \renewcommand{\arraystretch}{1.06}
    \resizebox{\textwidth}{!}{%
    \begin{tabular}{l rrr rrr rrr rrr}
      \toprule
      & \multicolumn{3}{c}{\textbf{CG-MIP}} & \multicolumn{3}{c}{\textbf{CG-DP}} & \multicolumn{3}{c}{\textbf{BDCG-MIP}} & \multicolumn{3}{c}{\textbf{BDCG-DP}} \\
      \cmidrule(lr){2-4}\cmidrule(lr){5-7}\cmidrule(lr){8-10}\cmidrule(lr){11-13}
      \textbf{$T$ (n)} & Median & Mean & {[Min, Max]} & Median & Mean & {[Min, Max]} & Median & Mean & {[Min, Max]} & Median & Mean & {[Min, Max]} \\
      \midrule
      \multicolumn{13}{l}{\emph{Phase~1 Time (min)}} \\
      \midrule
      $T_1$ (20) & 29.70 & 29.72 & {[3.39, 80.22]} & \textbf{1.30} & \textbf{1.57} & \textbf{{[0.74, 3.75]}} & 20.76 & 31.84 & {[3.98, 87.24]} & 1.36 & 1.63 & {[0.82, 2.85]} \\
      $T_2$ (18) & 40.27 & 41.82 & {[7.89, 76.72]} & \textbf{3.05} & \textbf{3.14} & \textbf{{[1.77, 4.96]}} & 45.78 & 48.24 & {[15.32, 83.54]} & 3.43 & 3.46 & {[2.10, 5.26]} \\
      $T_4$ (14) & 79.83 & 74.34 & {[34.30, 101.58]} & \textbf{8.15} & \textbf{8.07} & \textbf{{[6.08, 10.27]}} & 81.94 & 83.40 & {[46.42, 110.69]} & 10.02 & 10.13 & {[7.04, 13.59]} \\
      $T_8$ (6) & 157.45 & 159.00 & {[120.94, 188.73]} & \textbf{46.17} & \textbf{46.52} & \textbf{{[34.09, 60.19]}} & 177.23 & 186.24 & {[161.82, 225.86]} & 133.59 & 148.87 & {[95.21, 225.94]} \\
      \bottomrule
    \end{tabular}%
    }
    }
  \end{table*}

\begin{table*}[t]
    \centering
    \caption{Phase~2 wall-clock time (minutes) for CG-MIP, CG-DP, BDCG-MIP, and BDCG-DP.}
    \label{tab:ec_four_method_phase2}
    {\scriptsize
    \setlength{\tabcolsep}{2.8pt}
    \renewcommand{\arraystretch}{1.06}
    \resizebox{\textwidth}{!}{%
    \begin{tabular}{l rrr rrr rrr rrr}
      \toprule
      & \multicolumn{3}{c}{\textbf{CG-MIP}} & \multicolumn{3}{c}{\textbf{CG-DP}} & \multicolumn{3}{c}{\textbf{BDCG-MIP}} & \multicolumn{3}{c}{\textbf{BDCG-DP}} \\
      \cmidrule(lr){2-4}\cmidrule(lr){5-7}\cmidrule(lr){8-10}\cmidrule(lr){11-13}
      \textbf{$T$ (n)} & Median & Mean & {[Min, Max]} & Median & Mean & {[Min, Max]} & Median & Mean & {[Min, Max]} & Median & Mean & {[Min, Max]} \\
      \midrule
      \multicolumn{13}{l}{\emph{Phase~2 Time (min)}} \\
      \midrule
      $T_1$ (20) & 0.26 & 0.42 & {[0.11, 1.15]} & 0.22 & 0.40 & {[0.12, 1.17]} & 0.21 & 0.29 & {[0.12, 0.99]} & \textbf{0.20} & \textbf{0.28} & {[0.10, 1.03]} \\
      $T_2$ (18) & 2.81 & 8.67 & {[0.48, 93.40]} & 2.33 & 21.69 & {[0.52, 136.32]} & 2.66 & 2.86 & {[0.36, 6.10]} & \textbf{2.27} & \textbf{2.86} & {[0.39, 12.20]} \\
      $T_4$ (14) & 267.53 & 268.64 & {[157.40, 380.46]} & 318.21 & 321.95 & {[198.65, 471.84]} & 51.37 & 58.17 & {[14.09, 147.35]} & \textbf{30.51} & \textbf{33.43} & {[8.12, 75.77]} \\
      $T_8$ (6) & 322.54 & 320.98 & {[291.24, 359.04]} & 433.82 & 433.47 & {[419.79, 445.89]} & \textbf{279.94} & \textbf{285.35} & {[254.25, 318.27]} & 295.51 & 296.75 & {[195.61, 384.93]} \\
      \bottomrule
    \end{tabular}%
    }
    }
  \end{table*}

Figure~\ref{fig:ec_four_method_phase2_progress} complements Table~\ref{tab:ec_four_method_phase2} by plotting the feasible-objective progress from algorithm start on the six $T_8$ instances. This view is useful because the inherited-time protocol couples the raw Phase~2 minutes to the amount of time already spent in Phase~1. A method that enters integer refinement later begins its progress curve later on the wall-clock axis and inherits less remaining budget. The orange and red callouts mark the first feasible BDCG-MIP and BDCG-DP solutions on each instance and report the corresponding time, objective, and root-LP gap. These annotations make it easier to distinguish whether differences in raw Phase~2 time reflect the second-stage refinement itself or the earlier Phase~1 handoff.

The ablation tables confirm the distinct roles of the pricing solver and the decomposition architecture. Holding the generic MIP pricing fixed, moving from CG-MIP to BDCG-MIP yields the largest improvements on the longest horizons. At $T_8$, the decomposition architecture alone reduces the median root-LP gap from 28.11\% down to 4.40\%. Conversely, holding the BDCG architecture fixed, upgrading from MIP pricing to exact DP pricing yields its largest relative runtime gains on the short horizons. At $T_1$, BDCG-DP reduces the median end-to-end runtime from 21.23 minutes to 1.87 minutes compared to BDCG-MIP. The DP pricing upgrade also provides consistent benefits on longer horizons, reducing the median end-to-end runtime at $T_4$ from 140.68 to 42.34 minutes and slightly improving the median root-LP gap at $T_8$ from 4.40\% to 4.12\%. Together, these results validate the hybrid BDCG-DP design, which leverages exact DP pricing for column generation and Benders decomposition for long-horizon coordination.

\begin{figure*}[t]
  \centering
  \includegraphics[width=0.94\textwidth]{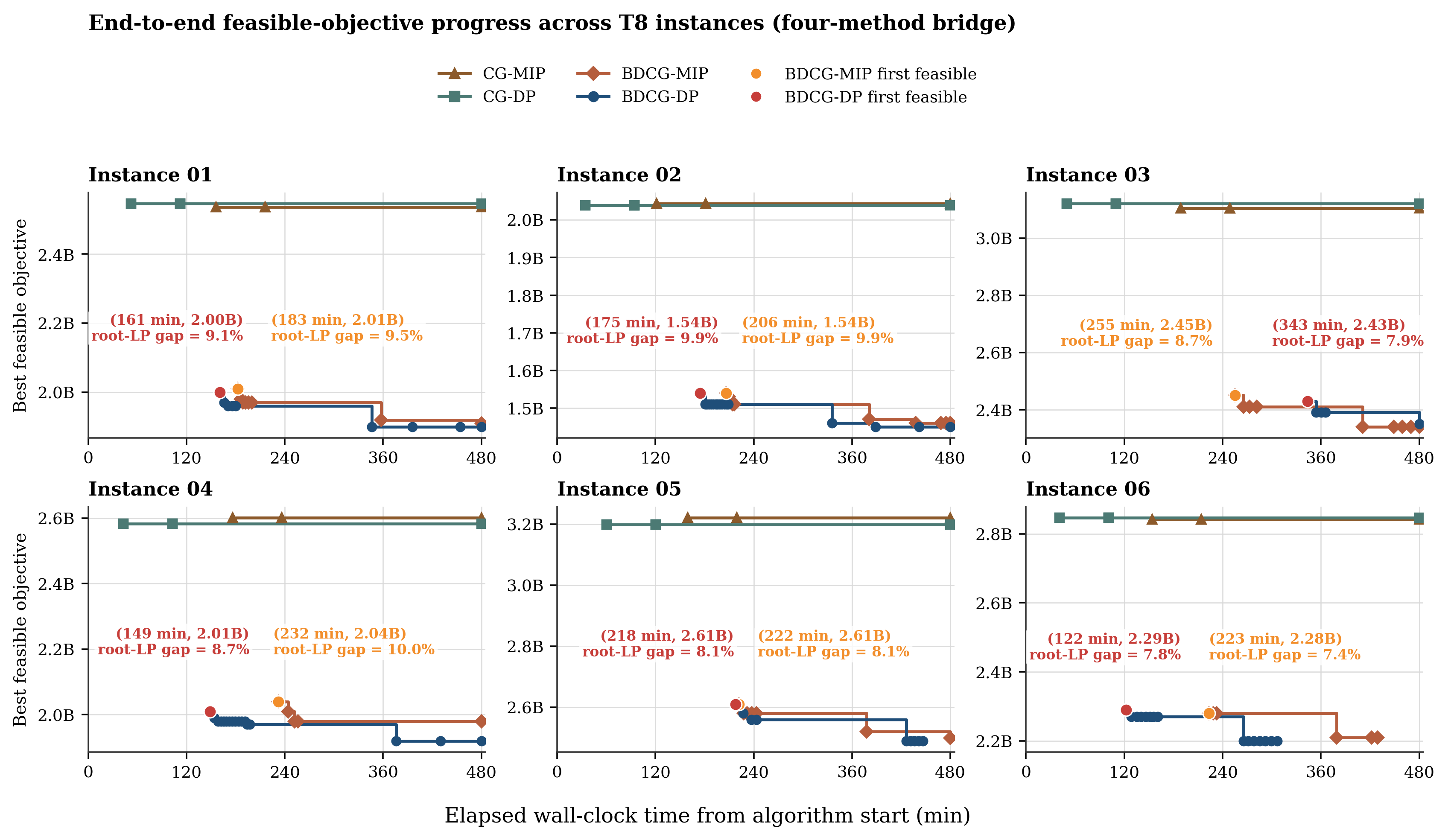}
  \caption{End-to-end feasible-objective progress on the six $T_8$ instances for CG-MIP, CG-DP, BDCG-MIP, and BDCG-DP. The orange and red callouts mark the first feasible BDCG-MIP and BDCG-DP solutions, respectively, and report the corresponding elapsed time, objective value, and root-LP gap.}
  \label{fig:ec_four_method_phase2_progress}
\end{figure*}

\subsection{Long-horizon Phase~1 timing and convergence diagnostics}
\label{sec:ec_long_horizon_diagnostics}

As the planning horizon expands, the distribution of computational effort within Phase~1 shifts. Table~\ref{tab:ec_phase1_breakdown} decomposes the median BDCG-DP Phase~1 time by component; the residual covers master reoptimization and related coordination work. At $T_4$, pricing is the largest component; at $T_8$, the residual dominates. This shift indicates that master reoptimization and related coordination work dominate the Phase~1 effort on the largest instances.

\begin{table}[!t]
  \centering
  \caption{Median BDCG-DP Phase~1 time decomposition (minutes)}
  \label{tab:ec_phase1_breakdown}
  {\small
  \begin{tabular}{lrrrrr}
    \toprule
    & \textbf{Pricing} & \textbf{Cut generation} & \textbf{Chain LPs} & \textbf{Master and other} & \textbf{Total} \\
    \midrule
    $T_4$ & 5.36 & 0.78 & 0.17 & 3.71 & 10.02 \\
    $T_8$ & 10.33 & 2.02 & 0.37 & 120.86 & 133.59 \\
    \bottomrule
  \end{tabular}}
\end{table}

Figure~\ref{fig:ec_phase1_diagnostics} complements that timing view with two additional Phase~1 diagnostics on the six $T_8$ instances. The top panels compare the CG-DP Phase~1 LP objective against the BDCG-DP Phase~1 lower and evaluated upper bounds. The bottom panels track the relative Phase~1 gap closure of BDCG-DP. Across all six instances, the BDCG-DP evaluated upper bound falls sharply early in Phase~1, and the BDCG-DP bound bracket then continues to close before the handoff to Phase~2. This pattern supports viewing Phase~1 not only as a pricing stage, but also as an important investment in improving the downstream approximation before integer refinement begins.

\begin{figure*}[t]
  \centering
  \includegraphics[width=0.92\textwidth]{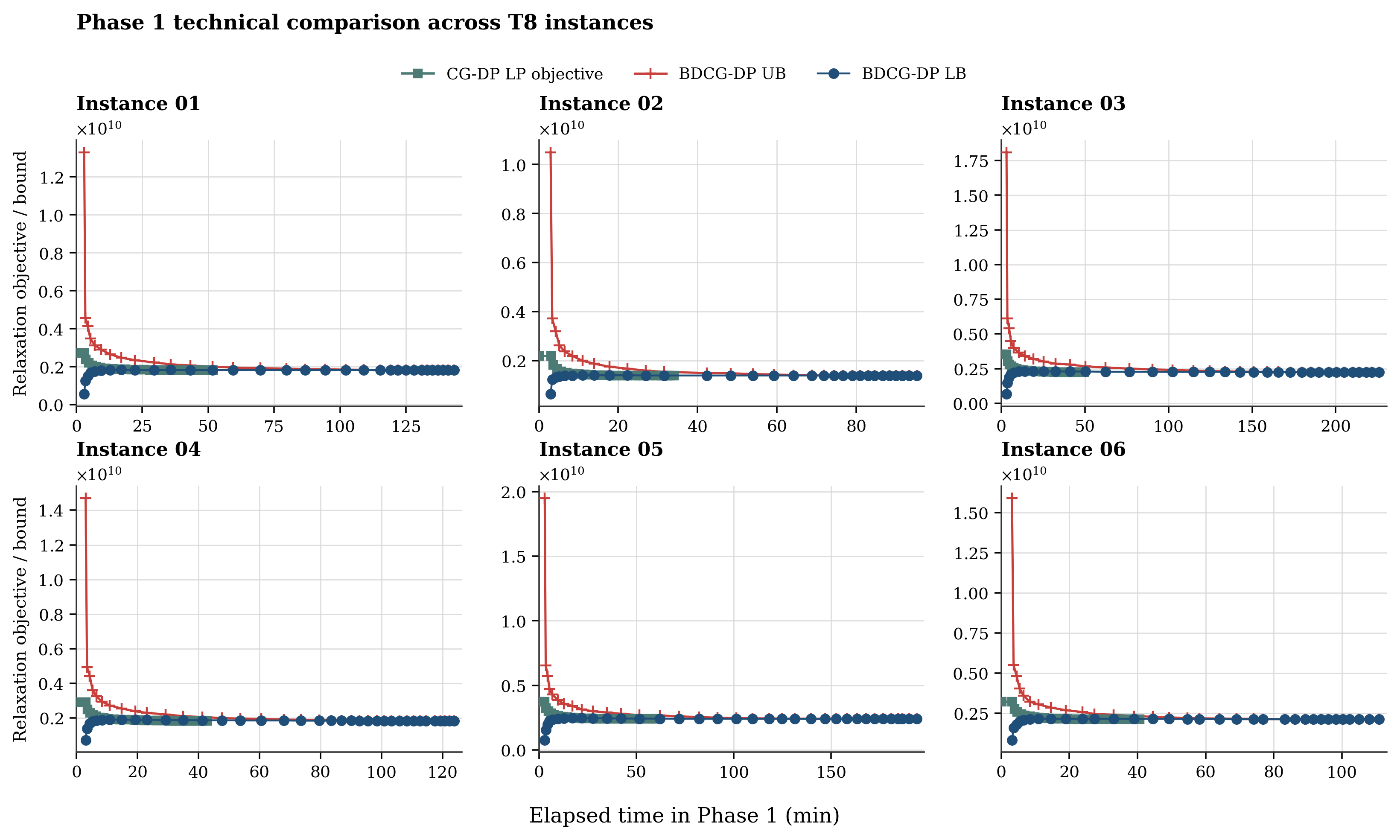}

  \vspace{0.8em}

  \includegraphics[width=0.92\textwidth]{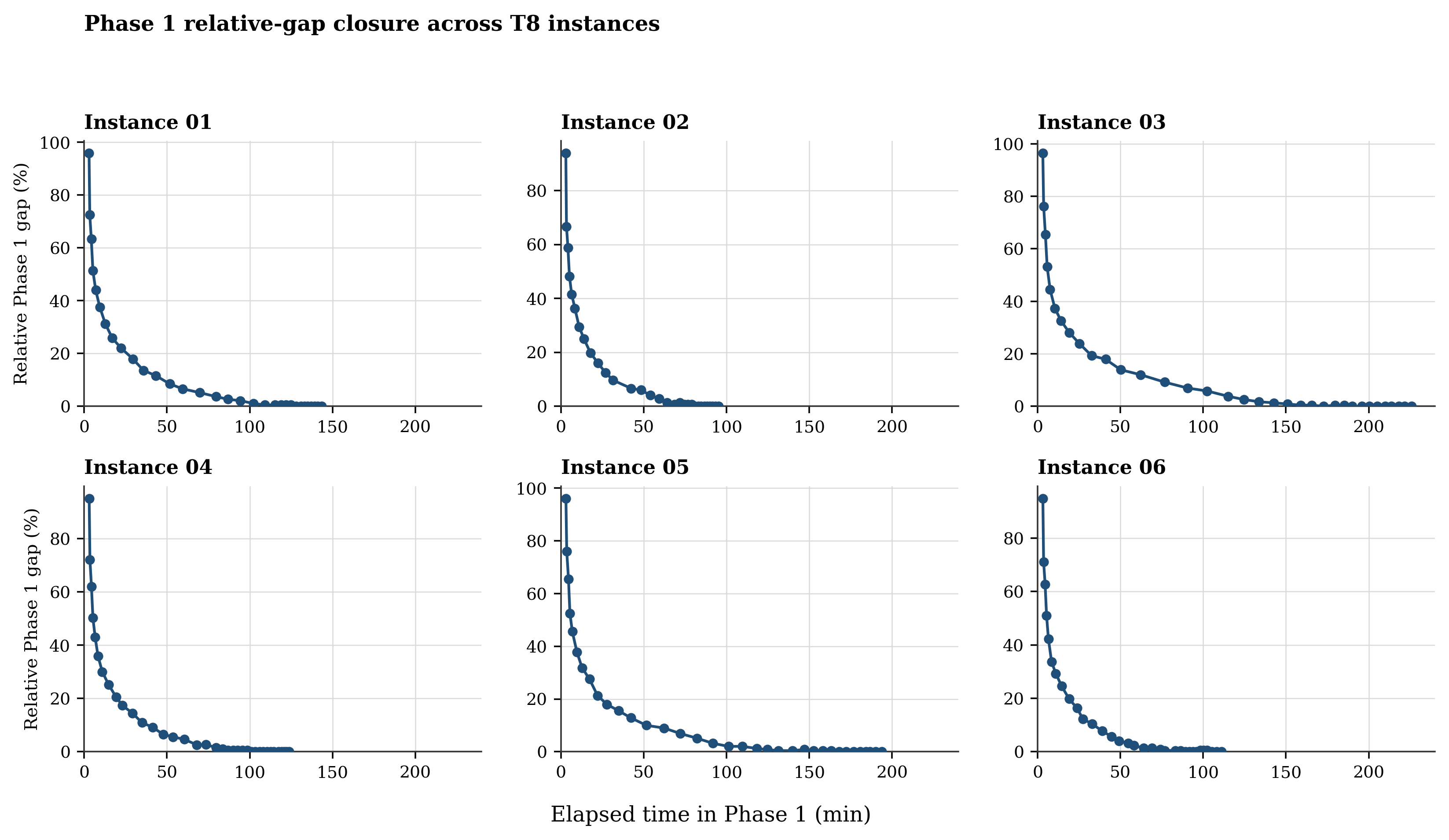}
  \caption{Phase~1 diagnostics on the six $T_8$ instances. Top: CG-DP and BDCG-DP Phase~1 trajectories, where the teal curve
  reports the CG-DP LP objective and the blue and red curves report the BDCG-DP lower and evaluated upper bounds. Bottom:
  relative Phase~1 gap closure under BDCG-DP, measured as the percentage gap between the evaluated upper and lower bounds over
  elapsed Phase~1 time.}
  \label{fig:ec_phase1_diagnostics}
\end{figure*}

\section{Phase~2 warm-start procedure}
\label{sec:ec_ph2_warmstart}

The Phase~2 initialization procedure used in the experiments is documented here. It is designed to improve the starting incumbent without altering the fixed-pool refinement model, the cut family, or the subsequent refinement logic. Because the same initialization is used throughout the BDCG family, it does not artificially inflate the performance of BDCG-DP relative to BDCG-MIP. The same staged initialization is also implemented, stage for stage, in the final integer solve of the CG-MIP and CG-DP baselines, so that all four methods enter their integer solve from comparably constructed incumbents. Comparisons between the CG and BDCG families therefore isolate the decomposition architecture at a matched initialization, rather than confounding it with an initialization advantage.

\textbf{Preparation.} At the start of Phase~2, all retained trimming-pattern variables $y^{\text{cut}}$ and loading-configuration variables $y^{\text{load}}$ are switched to integer type. The fixed-column master otherwise keeps the same production, transportation, and demand-side structure described in the main paper. The warm start changes only the initial incumbent and the amount of preliminary cut generation performed before the final integer-refinement loop begins.

\textbf{Stage 1: Mandatory-shipment baseline.} The first step solves a small auxiliary integer problem that covers the positive minimum-shipment rows in the master. This auxiliary model uses only trimming patterns and loading configurations that can contribute to those mandatory rows and minimizes production cost plus transportation cost. Its solution $(y^{\text{cut,base}}, y^{\text{load,base}})$ guarantees the mandatory shipments.

\textbf{Stage 2: Baseline accounting.} Given $y^{\text{cut,base}}$, the algorithm computes the run hours already implied by the mandatory-shipment baseline. This is not an optimization step. It simply records how much of the fixed Phase~1 column set is already committed before residual reoptimization is attempted.

\textbf{Stage 3: Residual LP on the full column set.} Next, a residual LP is solved over the full Phase~1 column set. The minimum-shipment rows are temporarily removed, but lower bounds are imposed so that the LP must keep at least the baseline activity. This LP can reallocate only the nonmandatory part of the solution across the full trimming-pattern and loading-configuration sets retained from Phase~1.

\textbf{Stage 4: Rounded implementable start.} The residual LP solution is then rounded asymmetrically. Trimming-pattern counts $y^{\text{cut}}$ are rounded down, whereas loading-configuration counts $y^{\text{load}}$ are rounded up. Rounding down trimming-pattern counts is conservative on production, while rounding up loading-configuration counts preserves the lane capacity needed to support the loaded rolls. After rounding, the loading-configuration counts are fixed temporarily.

\textbf{Preliminary Benders pass with fixed loading counts.} With the rounded $y^{\text{load}}$ values fixed, the algorithm runs a short Phase~2 Benders loop on the same master model until no violated demand-side cuts remain, or until a short time limit is reached. During this pass, only the trimming side is allowed to adjust. The goal is to improve the incumbent before the final solve by closing part of the demand-side cuts while keeping the loading side fixed.

\textbf{Final Phase~2 solve.} After this preliminary pass, the loading-configuration bounds are restored, and the last incumbent from the warm-start loop is used to initialize the full Phase~2 integer solve. The algorithm then executes the standard Phase~2 refinement loop described in the main paper.

\textbf{Summary of the initialization.} The warm start is designed to improve the first incumbent and reduce early branch-and-bound effort. It first guarantees the mandatory shipments, then rebalances the remaining activity over the full fixed column set, and finally performs a short cut-closure pass before the final integer solve begins. The underlying Phase~2 algorithm remains the standard fixed-column integer Benders loop.

\section{Proofs omitted from the main text}

Each theoretical result from the main text is restated before its proof. The exact-reformulation claim from
Section~\ref{sec:pattern_based_model} is also recorded in proof form.

\subsection{Supply-profile decomposition (Proposition~\ref{prop:supply_profile_decomposition})}

\noindent\textbf{Proposition~\ref{prop:supply_profile_decomposition} (restated).}
Let $\mathbf{S}=(\mathbf{S}_{p,c})_{p\in P,\ c\in C^{\mathrm{MTS}}}$ be a fixed feasible aggregate MTS supply profile, where $\mathbf{S}_{p,c}=(S_{p,c,t})_{t\in T}$. If $Z^{\mathrm{LP}}_{\mathrm{MTS}}(\mathbf{S})$ denotes the minimum downstream MTS LP cost for that profile, then
\[
  Z^{\mathrm{LP}}_{\mathrm{MTS}}(\mathbf{S})
  =
  \sum_{p\in P}\sum_{c\in C^{\mathrm{MTS}}} Z^{\mathrm{LP}}_{p,c}(\mathbf{S}_{p,c}).
\]

\medskip
\noindent\emph{Proof.}
Once $\mathbf{S}$ is fixed, the downstream MTS variables for one pair $(p,c)$ consist only of that pair's inventory, lost-sales, safety-stock-shortfall, and penalty variables over time. The corresponding flow-balance, safety-stock, and penalty constraints involve only the same pair's supply vector $\mathbf{S}_{p,c}$ and data. No downstream MTS constraint or objective term couples one pair $(p,c)$ to another. Therefore the downstream MTS LP separates into independent chain LPs indexed by $(p,c)\in P\times C^{\mathrm{MTS}}$, and its optimal value is the sum of the individual chain values. \Halmos

\subsection{Productwise reduction of trimming pricing (Lemma~\ref{lem:trim_productwise_reduction})}

\noindent\textbf{Lemma~\ref{lem:trim_productwise_reduction} (restated).}
For a fixed run $b$, the optimal value of trimming pricing problem \eqref{eq:bdcg_trim_pricing} equals the optimal value of
\[
  \max \left\{
    \sum_{p\in P} \bar\delta_{b,p}\,x_p
    \;:\;
    \sum_{p\in P} w_p x_p \le W_b,\;
    x_p\in\mathbb{Z}_+
  \right\},
\]
where $\bar\delta_{b,p}:=\max_{c\in C,\,v\in V}\delta_{b,p,c,v}$. Moreover, any optimal aggregate solution can be lifted to an optimal trimming pattern by assigning each product $p$ to any customer-lane pair attaining $\bar\delta_{b,p}$.

\medskip
\noindent\emph{Proof.}
Consider any feasible solution $a=(a_{p,c,v})$ of \eqref{eq:bdcg_trim_pricing}, and define aggregate product counts
\[
  x_p := \sum_{c\in C}\sum_{v\in V} a_{p,c,v},
  \qquad p\in P.
\]
Because all copies of product $p$ have the same width $w_p$, the aggregate vector $x=(x_p)$ satisfies
\[
  \sum_{p\in P} w_p x_p
  =
  \sum_{p\in P}\sum_{c\in C}\sum_{v\in V} w_p a_{p,c,v}
  \le W_b,
\]
so $x$ is feasible for the reduced productwise knapsack. Its objective value satisfies
\[
  \sum_{p\in P}\sum_{c\in C}\sum_{v\in V} \delta_{b,p,c,v}\,a_{p,c,v}
  \le
  \sum_{p\in P}\bar\delta_{b,p}\,x_p,
\]
because $\delta_{b,p,c,v}\le \bar\delta_{b,p}$ for every customer-lane pair. Hence every feasible pattern solution maps to a feasible aggregate solution with objective value at least as large.

Conversely, let $x$ be any feasible solution of the reduced productwise knapsack. For each product $p$, choose any pair $(c^\star_p,v^\star_p)$ attaining $\bar\delta_{b,p}$. Define
\[
  a_{p,c,v}
  :=
  \begin{cases}
    x_p, & (c,v)=(c^\star_p,v^\star_p),\\
    0,   & \text{otherwise}.
  \end{cases}
\]
This lifted pattern is feasible for \eqref{eq:bdcg_trim_pricing} because it uses the same aggregate counts $x_p$ and therefore the same total width. Its objective value is
\[
  \sum_{p\in P}\sum_{c\in C}\sum_{v\in V} \delta_{b,p,c,v}\,a_{p,c,v}
  =
  \sum_{p\in P} \bar\delta_{b,p}\,x_p.
\]
Thus every feasible aggregate solution can be lifted to a feasible pattern solution with the same objective value, and the two optimization problems have the same optimal value. The lifting construction also proves the final claim. \Halmos

\subsection{LP optimality of Phase~1 (Proposition~\ref{prop:ph1_lp_optimality})}

\noindent\textbf{Proposition~\ref{prop:ph1_lp_optimality} (restated).}
For each $(p,c)\in P\times C^{\mathrm{MTS}}$, let $\mathcal{T}_{p,c}$ denote the family of all supporting hyperplanes of $Z^{\mathrm{LP}}_{p,c}(\cdot)$. Suppose each Phase~1 restricted master LP is solved to optimality, the trimming and loading pricing subproblems are solved exactly, and demand-side separation returns a violated cut from $\mathcal{T}_{p,c}$ whenever $(\mathbf{S}_{p,c},\theta^{\mathrm{MTS}}_{p,c})$ violates the epigraph of $Z^{\mathrm{LP}}_{p,c}$. If Phase~1 terminates with no negative-reduced-cost trimming pattern, no negative-reduced-cost loading configuration, and no violated demand-side cut, then the current solution is optimal for the LP relaxation of the decomposed master over the full supply-side column families and the full cut families $\{\mathcal{T}_{p,c}\}$.

\medskip
\noindent\emph{Proof.}
For each $(p,c)$, let $\bar{\mathcal{T}}_{p,c}\subseteq \mathcal{T}_{p,c}$ denote the cut pool present at Phase~1 termination, and let $(\bar y,\bar r,\bar S,\bar\theta)$ denote the terminating solution. Because the terminating restricted master LP is solved to optimality, $(\bar y,\bar r,\bar S,\bar\theta)$ is optimal for the LP master defined by the current column pools and the current cut pools.

Exact pricing and the absence of negative-reduced-cost trimming patterns or loading configurations imply the standard Dantzig--Wolfe optimality condition: the same solution is optimal for the LP master obtained by replacing the restricted supply-side column pools with the full families $K_b$ and $\Gamma_{b,c,v}$ while keeping the cut pools fixed at $\bar{\mathcal{T}}_{p,c}$.

Next consider the cut families. Because demand-side separation is exact and no violated cut remains at termination, we have
\[
  \bar\theta^{\mathrm{MTS}}_{p,c} \ge Z^{\mathrm{LP}}_{p,c}(\bar{\mathbf{S}}_{p,c})
  \qquad \forall\, (p,c)\in P\times C^{\mathrm{MTS}}.
\]
Therefore $(\bar y,\bar r,\bar S,\bar\theta)$ lies in the epigraph of each value function $Z^{\mathrm{LP}}_{p,c}$, and hence it satisfies every supporting hyperplane in $\mathcal{T}_{p,c}$. In particular, it is feasible for the full decomposed LP master with the complete cut families $\{\mathcal{T}_{p,c}\}$.

Finally, the LP master with the complete cut families is a restriction of the LP master with the same full column families and only the smaller cut pools $\bar{\mathcal{T}}_{p,c}$. Any feasible solution of the former is therefore feasible for the latter. Since $(\bar y,\bar r,\bar S,\bar\theta)$ is already optimal for the latter, no feasible solution of the full decomposed LP master can have a strictly smaller objective value. Hence $(\bar y,\bar r,\bar S,\bar\theta)$ is optimal for the LP relaxation of the decomposed master over the full supply-side column families and the full cut families $\{\mathcal{T}_{p,c}\}$. \Halmos

\subsection{Exact reformulation of the pattern-based model}

\noindent\textbf{Exactness claim.}
The integer pattern-based formulation \eqref{form:MP} is an exact reformulation of the item-based model. Every feasible item-based solution can be aggregated into a feasible integer pattern-based solution with the same objective value, and every feasible integer pattern-based solution can be disaggregated into roll-level cutting decisions and vehicle-level loading decisions with the same objective value.

\medskip
\noindent For the finite-index exactness argument, the companion makes explicit one auxiliary lane-level cap that is omitted from the main-text display because it is a technical finiteness device rather than part of the shipping logic:
\begin{equation}
\label{eq:ec_pat_load_count}
  \sum_{\gamma\in\Gamma_{b,c,v}} y^{\text{load}}_{b,c,v,\gamma}
  \le |J_{b,c,v}|,
  \qquad \forall\, b,c,v.
\end{equation}

Because $\mathrm{cost}^{\text{veh}}_{b,c,v}>0$, at least one optimal pattern-based solution dispatches no vehicle whose load can be removed, so imposing \eqref{eq:ec_pat_load_count} does not change the optimal value; the cap is omitted from the implemented master.

\medskip
\noindent\emph{Proof.}
Consider first any feasible item-based solution. For each produced jumbo roll $o\in O_b$ with $u_{b,o}=1$, define its cut vector
\[
  a_{p,c,v}(o):=x_{o,p,c,v}.
\]
By \eqref{aix_eq:deckle_cap}, this vector is width-feasible for run $b$, so it is a feasible trimming pattern for that run. Aggregating identical roll-level cut vectors over all produced rolls in run $b$ yields integer pattern counts $y^{\text{cut}}_{b,k}$. Likewise, for each dispatched vehicle $j\in J_{b,c,v}$ with $z_{b,c,v,j}=1$, define its load vector
\[
  a^{\text{load}}_{p}(j):=\chi_{b,p,c,v,j}.
\]
By \eqref{aix_eq:veh_wt_ind}--\eqref{aix_eq:veh_pos_ind}, this vector is a feasible loading configuration on lane $(b,c,v)$. Aggregating identical vehicle-load vectors yields integer counts $y^{\text{load}}_{b,c,v,\gamma}$. The run-time constraints, minimum-shipment constraints, and supply-profile equations are preserved by construction, and \eqref{eq:pat_coupling} follows from the item-based linking equalities \eqref{aix_eq:load_link}. Production cost, transportation cost, lost-sales cost, and safety-stock penalty cost are unchanged, so the aggregated pattern solution is feasible with the same objective value.

Conversely, consider any feasible integer pattern-based solution. For each run $b$ and each pattern $k$ used with multiplicity $y^{\text{cut}}_{b,k}$, create that many explicit jumbo rolls in $O_b$, set $u_{b,o}=1$ on those created rolls and $u_{b,o}=0$ on all remaining roll indices, assign to each created roll the cut quantities
\[
  x_{o,p,c,v}:=a_{p,c,v,k}.
\]
For unused roll indices, set $x_{o,p,c,v}=0$ for all $(p,c,v)$.
This is feasible because each pattern in $K_b$ satisfies the deckle restriction. Moreover, every feasible pattern-based solution satisfies the inherited run upper bound $r_b\le r_b^{\max}$, so \eqref{eq:pat_run_vs_patterns} implies
\[
  h_b\sum_{k\in K_b} y^{\text{cut}}_{b,k} \le r_b \le r_b^{\max}.
\]
Hence
\[
  \sum_{k\in K_b} y^{\text{cut}}_{b,k} \le \left\lfloor \frac{r_b^{\max}}{h_b}\right\rfloor = |O_b|,
\]
which guarantees that enough potential roll indices are available in $O_b$. Similarly, for each lane $(b,c,v)$ and each configuration $\gamma$ used with multiplicity $y^{\text{load}}_{b,c,v,\gamma}$, create that many explicit vehicles in $J_{b,c,v}$, set $z_{b,c,v,j}=1$ for those created vehicles and $z_{b,c,v,j}=0$ for all remaining indices, and define componentwise capacity vectors
\[
  \bar\chi_{b,p,c,v,j}:=a^{\text{load}}_{b,p,c,v,\gamma}.
\]
Each such vector is feasible by definition of $\Gamma_{b,c,v}$, and the auxiliary cap \eqref{eq:ec_pat_load_count} ensures that the required multiplicities fit within the available vehicle indices in $J_{b,c,v}$. Because \eqref{eq:pat_coupling} guarantees that the aggregate cut quantity of every product on every lane does not exceed the aggregate reserved product capacity on that lane, the actual cut quantities can be allocated across the created vehicles so that
\[
  0 \le \chi_{b,p,c,v,j} \le \bar\chi_{b,p,c,v,j}
\]
for every product and vehicle, and
\[
  \sum_{j\in J_{b,c,v}} \chi_{b,p,c,v,j}
  =
  \sum_{o\in O_b} x_{o,p,c,v}
\]
for every $(b,p,c,v)$. One constructive way to obtain such an allocation is to assign each product independently across the created vehicles on lane $(b,c,v)$, processing the vehicles in any order and loading up to the remaining reserved capacity $\bar\chi_{b,p,c,v,j}$ for that product; the aggregate inequality \eqref{eq:pat_coupling} guarantees that this greedy assignment exhausts the required quantity. For all noncreated vehicle indices, set $\chi_{b,p,c,v,j}=0$ for every product $p$. Any componentwise reduction of a feasible loading vector remains feasible under nonnegative weight and position coefficients, so the explicit vehicle loads satisfy \eqref{aix_eq:load_link} and the item-based capacity constraints. Carrying over the same run-time, lost-sales, inventory, and penalty variables yields a feasible item-based solution with the same objective value. \Halmos

\subsection{Integrality on integral supply profiles (Lemma~\ref{lem:demand_integrality})}

\noindent\textbf{Lemma~\ref{lem:demand_integrality} (restated).}
Fix a customer-product pair $(p,c)$. Suppose the initial inventory $I^{\text{init}}_{p,c}$, demands $d_{p,c,t}$, safety-stock targets $ss_{p,c,t}$, and the convex piecewise-linear safety-stock penalty admit an equivalent breakpoint representation with integral breakpoints. Then, for any integral supply vector $\mathbf{S}_{p,c}$, subproblem \eqref{model:demand_primal} admits an equivalent incremental linear formulation with a totally unimodular constraint matrix and an integral right-hand side. Consequently, if $Z^{\mathrm{IP}}_{p,c}(\mathbf{S}_{p,c})$ denotes the value of the same chain problem under the original discrete inventory domains of $I^{+}_{p,c,t}$, $l_{p,c,t}$, and $I^{-}_{p,c,t}$, then $Z^{\mathrm{LP}}_{p,c}(\mathbf{S}_{p,c}) = Z^{\mathrm{IP}}_{p,c}(\mathbf{S}_{p,c})$.

\medskip
\noindent\emph{Proof.}
Within this proof only, write $I_t:=I^{+}_{p,c,t}$, $l_t:=l_{p,c,t}$, and $x_t:=I^{-}_{p,c,t}$ for brevity. Let
\[
  \tilde\varphi_t(x)
  :=
  \max\!\Bigl\{0,\max_{i\in N}\{\alpha_{p,c,t,i}x-\beta_{p,c,t,i}\}\Bigr\},
  \qquad x\ge 0,
\]
denote the convex piecewise-linear safety-stock penalty induced by \eqref{eq:demand_cost} and
\eqref{eq:demand_vars} in period $t$. Because $\psi_{p,c,t}$ is minimized in \eqref{model:demand_primal}, every optimal
solution satisfies $\psi_{p,c,t}=\tilde\varphi_t(x_t)$. The variable $\psi_{p,c,t}$ can
therefore be eliminated from the proof without changing the optimal value.
Under the lemma's assumption, this same function admits an equivalent breakpoint representation with integral breakpoints.

From the flow-balance constraints in \eqref{model:demand_primal}, the end-of-period inventory can be written recursively as
\[
  I_t
  =
  I^{\text{init}}_{p,c}
  + \sum_{\tau=1}^{t}\bigl(S_{p,c,\tau}-d_{p,c,\tau}+l_{\tau}\bigr),
  \qquad t\in T.
\]
Hence the inventory nonnegativity constraints $I_t\ge 0$ are equivalent to
\begin{equation}
\label{eq:ec_cum_inv_nonneg}
  \sum_{\tau=1}^{t} l_{\tau}
  \;\ge\;
  \sum_{\tau=1}^{t} d_{p,c,\tau}
  - I^{\text{init}}_{p,c}
  - \sum_{\tau=1}^{t} S_{p,c,\tau},
  \qquad t\in T.
\end{equation}
Similarly, the safety-stock constraints $x_t + I_t \ge ss_{p,c,t}$ are equivalent to
\begin{equation}
\label{eq:ec_cum_ss}
  \sum_{\tau=1}^{t} l_{\tau} + x_t
  \;\ge\;
  ss_{p,c,t}
  + \sum_{\tau=1}^{t} d_{p,c,\tau}
  - I^{\text{init}}_{p,c}
  - \sum_{\tau=1}^{t} S_{p,c,\tau},
  \qquad t\in T.
\end{equation}

Now consider the convex piecewise-linear function $\tilde\varphi_t$. Let its integral breakpoints, i.e., the points at
which its slope changes, be
\[
  0=\bar b_{t,0}<\bar b_{t,1}<\cdots<\bar b_{t,m_t},
\]
and let $s_{t,r}$ denote the slope on interval $[\bar b_{t,r-1},\bar b_{t,r}]$ for $r=1,\dots,m_t$, with
$s_{t,m_t+1}$ denoting the slope on the final ray $[\bar b_{t,m_t},\infty)$. Convexity implies that these slopes are
nondecreasing. Introduce incremental variables $z_{t,r}$ such that
\[
  x_t = \sum_{r=1}^{m_t+1} z_{t,r},
\]
with
\[
  0\le z_{t,r}\le \bar b_{t,r}-\bar b_{t,r-1}
  \quad (r=1,\dots,m_t),
  \qquad
  z_{t,m_t+1}\ge 0,
\]
where $z_{t,m_t+1}$ represents the final unbounded ray if the penalty extends beyond the last breakpoint. In this
representation,
\[
  \tilde\varphi_t(x_t)
  =
  c_t + \sum_{r=1}^{m_t+1} s_{t,r}\,z_{t,r},
\]
where $c_t:=\tilde\varphi_t(0)$ is constant with respect to the decision variables. Substituting
$x_t=\sum_{r=1}^{m_t+1} z_{t,r}$ into \eqref{eq:ec_cum_ss} yields
\begin{equation}
\label{eq:ec_incremental_ss}
  \sum_{\tau=1}^{t} l_{\tau} + \sum_{r=1}^{m_t+1} z_{t,r}
  \;\ge\;
  ss_{p,c,t}
  + \sum_{\tau=1}^{t} d_{p,c,\tau}
  - I^{\text{init}}_{p,c}
  - \sum_{\tau=1}^{t} S_{p,c,\tau},
  \qquad t\in T.
\end{equation}
Therefore, up to the additive constant $\sum_{t\in T} c_t$, subproblem \eqref{model:demand_primal} is equivalent to the
following incremental linear program:
\begin{align}
  \min \quad
  & \sum_{t\in T} \mathrm{cost}^{\text{lost}}_{p,c}\,l_t
    + \sum_{t\in T}\sum_{r=1}^{m_t+1} s_{t,r}\,z_{t,r}
    \label{eq:ec_incremental_obj} \\
  \text{s.t.}\quad
  & \eqref{eq:ec_cum_inv_nonneg},\ \eqref{eq:ec_incremental_ss}, \nonumber\\
  & 0 \le l_t \le d_{p,c,t}
    && \forall\, t\in T, \label{eq:ec_incremental_l_bounds}\\
  & 0 \le z_{t,r} \le \bar b_{t,r}-\bar b_{t,r-1}
    && \forall\, t\in T,\ r=1,\dots,m_t, \label{eq:ec_incremental_z_bounds}\\
  & z_{t,m_t+1}\ge 0
    && \forall\, t\in T. \label{eq:ec_incremental_last_ray}
\end{align}

The right-hand sides in \eqref{eq:ec_cum_inv_nonneg} and \eqref{eq:ec_incremental_ss} are integral because
$I^{\text{init}}_{p,c}$, $d_{p,c,t}$, $ss_{p,c,t}$, and the supply profile $\mathbf{S}_{p,c}$ are integral by assumption.
The upper bounds in \eqref{eq:ec_incremental_l_bounds}--\eqref{eq:ec_incremental_z_bounds} are also integral.

It remains to show total unimodularity. Consider the coefficient matrix of the cumulative constraints
\eqref{eq:ec_cum_inv_nonneg} and \eqref{eq:ec_incremental_ss} after multiplying those rows and the columns for $l_t$ and
$z_{t,r}$ by $-1$, which preserves total unimodularity and produces a $0$--$1$ matrix. Order the rows as follows:
\[
  \eqref{eq:ec_cum_inv_nonneg}\text{ for }t=1,\quad
  \eqref{eq:ec_incremental_ss}\text{ for }t=1,\quad
  \eqref{eq:ec_cum_inv_nonneg}\text{ for }t=2,\quad
  \eqref{eq:ec_incremental_ss}\text{ for }t=2,\ \dots
\]
Under this ordering, each lost-sales variable $l_{\tau}$ has coefficient $+1$ in every row associated with periods
$t\ge \tau$ and coefficient $0$ before that point. Hence the nonzero entries in the column of $l_{\tau}$ form one consecutive
block. Each incremental variable $z_{t,r}$ appears only in the safety-stock row for period $t$, so its column also has a
single consecutive block of nonzero entries. Thus the resulting $0$--$1$ coefficient matrix has the consecutive-ones property
for columns and is therefore totally unimodular. Appending the upper-bound rows
\eqref{eq:ec_incremental_l_bounds}--\eqref{eq:ec_incremental_z_bounds} only adds unit rows, which preserves total
unimodularity, while the nonnegativity restrictions are standard sign constraints on the variables.

Because the incremental formulation has a totally unimodular constraint matrix and an integral right-hand side, it has an
integral optimal extreme point. In particular, the optimal values of $l_t$ and all incremental variables $z_{t,r}$ can be
chosen integral. It follows that $x_t=\sum_r z_{t,r}$ is integral for every $t$. Finally, the recovered inventory sequence
\[
  I_t
  =
  I^{\text{init}}_{p,c}
  + \sum_{\tau=1}^{t}\bigl(S_{p,c,\tau}-d_{p,c,\tau}+l_{\tau}\bigr)
\]
is integral as well. Therefore the original chain subproblem admits an optimal solution in which $I_t$, $l_t$, and $x_t$ are
all integral. \Halmos

\subsection{Finite termination of Phase~2 refinement (Proposition~\ref{prop:ph2_finite_termination})}

\noindent For the finite-state argument, the companion likewise makes explicit the omitted fixed-pool lane-level cap
\begin{equation}
\label{eq:ec_rmp_load_count}
  \sum_{\gamma\in\bar{\Gamma}_{b,c,v}} y^{\text{load}}_{b,c,v,\gamma}
  \le |J_{b,c,v}|,
  \qquad \forall\, b,c,v,
\end{equation}
which is suppressed in the main-text RMP display for readability. The implemented Phase-2 master also omits this cap; because $\mathrm{cost}^{\text{veh}}_{b,c,v}>0$, every optimal integer master solution satisfies \eqref{eq:ec_rmp_load_count}, since a vehicle whose load is not required by the coupling constraints can be removed at a cost saving. The finiteness argument below therefore applies to the incumbents encountered by the loop.

\medskip
\noindent\textbf{Proposition~\ref{prop:ph2_finite_termination} (restated).}
Assume Phase~2 is solved over fixed column pools and that each integer master solve is carried out to optimality. Under \eqref{eq:RMP_bound}, \eqref{eq:RMP_prod}, and \eqref{eq:ec_rmp_load_count}, the retained integer column-usage variables then range over a finite set. Assume also that every chain LP \eqref{model:demand_primal} is feasible and bounded for all $\mathbf{S}_{p,c}\ge 0$. If at each iteration $\ell$ at least one tangent cut of the form \eqref{eq:benders_optimality_cut} is added for every violated chain, then the Phase~2 refinement loop terminates in finitely many iterations (up to tolerance $\varepsilon_{\mathrm{viol}}$).

\medskip
\noindent\emph{Proof.}
Because the column pools are fixed, the trimming-pattern counts are bounded by \eqref{eq:RMP_bound} and \eqref{eq:RMP_prod}, while the loading-configuration counts are bounded by \eqref{eq:ec_rmp_load_count}. Since the retained column pools are finite and the corresponding variables are integer, only finitely many integer column-usage vectors can occur. Denote this finite set of reachable optimal integer incumbents by $\mathcal{Y}$, with $|\mathcal{Y}|<\infty$.

At each iteration $\ell$, the integer master returns an incumbent $y^\ell\in\mathcal{Y}$ that induces supply profiles
$\hat{\mathbf{S}}^\ell_{p,c}$ and epigraph values $\hat\theta^{\mathrm{MTS},\ell}_{p,c}$. If chain $(p,c)$ is violated (i.e.,
$Z^{\mathrm{LP}}_{p,c}(\hat{\mathbf{S}}^\ell_{p,c})-\hat\theta^{\mathrm{MTS},\ell}_{p,c}>\varepsilon_{\mathrm{viol}}$), a tangent cut is added
at $\hat{\mathbf{S}}^\ell_{p,c}$:
\[
  \theta^{\mathrm{MTS}}_{p,c}
  \;\ge\;
  Z^{\mathrm{LP}}_{p,c}(\hat{\mathbf{S}}^\ell_{p,c})
  + \pi^\top(\mathbf{S}_{p,c}-\hat{\mathbf{S}}^\ell_{p,c}).
\]
When evaluated at the same supply profile $\hat{\mathbf{S}}^\ell_{p,c}$, this cut enforces
$\theta^{\mathrm{MTS}}_{p,c}\ge Z^{\mathrm{LP}}_{p,c}(\hat{\mathbf{S}}^\ell_{p,c})$. Therefore, if the same integer solution $y^\ell$ (and hence
the same $\hat{\mathbf{S}}^\ell_{p,c}$) were returned in a later iteration, the current collection of cuts would already contain this tangent cut,
and the chain $(p,c)$ could not be violated at the same profile by more than $\varepsilon_{\mathrm{viol}}$.

Since each violating incumbent $y^\ell$ produces at least one new tangent cut that precludes re-acceptance of the same $(y^\ell,
c,p)$ violation pair, and $|\mathcal{Y}|\cdot|C^{\mathrm{MTS}}\times P|$ bounds the number of such distinct violating events,
the algorithm can encounter only finitely many of them before all chains satisfy
$\theta^{\mathrm{MTS}}_{p,c}\ge Z^{\mathrm{LP}}_{p,c}(\mathbf{S}_{p,c})-\varepsilon_{\mathrm{viol}}$ for every reachable integer
incumbent. \Halmos

\end{document}